\documentclass[leqno]{amsart}
\usepackage[T1]{fontenc}
\usepackage[latin9]{inputenc}
\usepackage{verbatim}
\usepackage{amstext}
\usepackage{amsthm}
\usepackage{amssymb}

\makeatletter
\numberwithin{equation}{section}
\numberwithin{figure}{section}
\theoremstyle{plain}
\newtheorem*{thm*}{\protect\theoremname}
\theoremstyle{plain}
\newtheorem{thm}{\protect\theoremname}
\theoremstyle{plain}
\newtheorem{lem}[thm]{\protect\lemmaname}
\theoremstyle{plain}
\newtheorem{prop}[thm]{\protect\propositionname}

\@ifundefined{date}{}{\date{}}

\providecommand{\lemmaname}{Lemma}

\providecommand{\theoremname}{Theorem}

\makeatother

\providecommand{\lemmaname}{Lemma}
\providecommand{\propositionname}{Proposition}
\providecommand{\theoremname}{Theorem}

\begin{document}
\title{Local and global analyticity for a generalized Camassa-Holm system}
\author{Hideshi Yamane}
\curraddr{{\small{}{}Department of Mathematical Sciences, Kwansei Gakuin University
}\\
{\small{}{}Gakuen 2-1 Sanda, Hyogo 669-1337, Japan}}
\email{{\small{}{}yamane@kwansei.ac.jp}}
\begin{abstract}
We solve the analytic Cauchy problem for the generalized two-component
Camassa-Holm system introduced by R. M. Chen and Y. Liu. We show the
existence of a unique local/global-in-time analytic solution under
certain conditions. This is the first result about global analyticity
for a Camassa-Holm-like system. The method of proof is basically that
developed by Barostichi, Himonas and Petronilho. The main differences
of their proof and ours are twofold: \\
(i) the system of Chen and Liu is not symmetic in the two unknowns
and our estimates are not trivial generalization of those in their
articles,  \\
(ii) we have simplified their argument by using fewer function spaces
and the main result is stated in a simple and natural way.
\end{abstract}


\maketitle
\markboth{Hideshi Yamane}{Camassa-Holm equations}\tableofcontents

\section*{Introduction}

We consider the generalized two-component CH system
\begin{equation}
\begin{cases}
 & m_{t}-\alpha u_{x}+\beta(2mu_{x}+um_{x})+3(1-\beta)uu_{x}+\rho\rho_{x}=0,\,m=u-u_{xx},\\
 & \rho_{t}+(\rho u)_{x}=0.
\end{cases}\label{eq:2CHrho}
\end{equation}

It was introduced in \cite{ChenLiu} as a model of shallow water waves
and is equivalent to 
\begin{equation}
\begin{cases}
 & u_{t}-u_{txx}-\alpha u_{x}+3uu_{x}-\beta(2u_{x}u_{xx}+uu_{xxx})+\rho\rho_{x}=0,\\
 & \rho_{t}+(\rho u)_{x}=0.
\end{cases}\label{eq:2CHrho-1}
\end{equation}
Here it is assumed that $u\to0$ and $\rho\to1$ hold as $|x|\to\infty$.
It is natural to introduce $v=\rho-1$, which tends to 0 as $|x|\to0$.
The system \eqref{eq:2CHrho} with $u\to0,\rho\to1$ is equivalent
to
\begin{equation}
\begin{cases}
 & (1-\partial_{x}^{2})u_{t}-\alpha u_{x}+3uu_{x}-\beta(2u_{x}u_{xx}+uu_{xxx})+(1+v)v_{x}=0,\\
 & v_{t}+u_{x}+(uv)_{x}=0,
\end{cases}\label{eq:2CHr}
\end{equation}
with  $u\to0$ , $v\to0$ as $|x|\to\infty$. Applying $(1-\partial_{x}^{2})^{-1}$
to the first equation, we get, by using $\partial_{x}^{2}(uu_{x})=uu_{xxx}+3u_{x}u_{xx}$,
\begin{equation}
\begin{cases}
 & u_{t}+\beta uu_{x}+(1-\partial_{x}^{2})^{-1}\partial_{x}\left[-\alpha u+\dfrac{3-\beta}{2}u^{2}+\dfrac{\beta}{2}u_{x}^{2}+v+\dfrac{1}{2}v^{2}\right]=0,\\
 & v_{t}+u_{x}+(uv)_{x}=0.
\end{cases}\label{eq:2CHr-1}
\end{equation}
In the sections below, we shall mainly consider \eqref{eq:2CHr-1}
rather than \eqref{eq:2CHrho-1} and \eqref{eq:2CHr}.

The main result of \cite{ChenLiu} is the following global unique
solvability theorem in the Sobolev spaces. 
\begin{thm*}
(\cite[Theorems 3.1, 5.1]{ChenLiu}) Assume $0<\beta<2$, $s>3/2$.
If $(u_{0},v_{0})\in H^{s}(\mathbb{\mathbb{R}})\times H^{s-1}(\mathbb{\mathbb{R}})$
and $\inf_{x\in\mathbb{R}}v_{0}(x)>-1$, then the Cauchy problem for
the system \eqref{eq:2CHr-1} with $u(0,x)=u_{0}$, $v(0,x)=v_{0}$
has a unique solution $(u,v)$ in the space $\mathcal{C}([0,\infty),H^{s}(\mathbb{\mathbb{R}})\times H^{s-1}(\mathbb{\mathbb{R}}))\cap\mathcal{C}^{1}([0,\infty),H^{s-1}(\mathbb{\mathbb{R}})\times H^{s-2}(\mathbb{\mathbb{R}}))$. 
\end{thm*}
In the present paper, we consider this Cauchy problem in the analytic
category. For $r>0$, set $S(r)=\left\{ x+iy\in\mathbb{C};\,|y|<r\right\} $
and 
\begin{align*}
A(r)= & \left\{ f\colon\mathbb{R}\to\mathbb{R};\,f(z)\text{ can be analytically continued to }S(r)\right\} \\
 & \cap\left\{ f\in L_{x,y}^{2}(S(r'))\text{ for all }0<r'<r\right\} .
\end{align*}
Then our main result is the following. 
\begin{thm*}
Assume $0<\beta<2$. If $u_{0},v_{0}\in A(r_{0})$ for some $r_{0}>0$
and $\inf_{x\in\mathbb{R}}v_{0}(x)>-1$, then the solution $(u,v)$
in the theorem above belongs to $\oplus^{2}\mathcal{C}^{\omega}([0,\infty)_{t}\times\mathbb{R}_{x})$.
Moreover, there exists a continuous function $\sigma(t)$ such that
we have $u(\cdot,t)$ and $v(\cdot,t)$ belong to $A(e^{\sigma(t)})$
for $t\in[0,T]$. 
\end{thm*}
In the course of the proof, we derive a local unique solvability result
of the Cauchy-Kowalevsky type, which is interesting in its own right. 

We recall some background. The original Camassa-Holm equation
\begin{equation}
u_{t}-u_{xxt}+3uu_{x}=2u_{x}u_{xx}+uu_{xxx}\label{eq:CH}
\end{equation}
was first proposed in \cite{FokasFuchssteiner} in the context of
integrability and was later studied in \cite{CamassaHolm}. It has
been found that \eqref{eq:CH} describes shallow water waves and other
physical phenomena. See \cite{ChenLiu}. The single equation \eqref{eq:CH}
admits many multi-component generalizations. A well-known one is 
\begin{equation}
\begin{cases}
 & m_{t}-Au_{x}+2mu_{x}+um_{x}+\rho\rho_{x}=0,\,m=u-u_{xx},\\
 & \rho_{t}+(\rho u)_{x}=0,
\end{cases}\label{eq:2CHrho-2}
\end{equation}
in \cite{Shabat-Alonso} and it corresponds to the case of $\beta=1$
in \eqref{eq:2CHrho-1}. 

The Camassa-Holm equation and its variants, including the ones above,
are studied from several points of view. Some authors employ the inverse
scattering technique (e.g. \cite{Boutet,Constantin}), while others
apply Hirota's bilinear method (e.g. \cite{Parker}). In a significant
number of articles, including \cite{ChenLiu}, solutions have been
obtained in the Sobolev or Besov spaces by using PDE techniques. Relatively
fewer number of articles deal with real- or complex-analytic solutions.
The analyticity of solutions is relevant since the equations model
water waves. See \cite{Constantin-Escher} for some discussion about
analyticity of water waves. In \cite{BHP,BHPpower,BHPglobal}, the
authors solved the Camassa-Holm and other equations in some spaces
of analytic solutions locally or globally. Their methods were employed
in \cite{yamane muCH} by the present author to solve the $\mu$-Camassa-Holm
and similar equations. In the present paper, we use similar methods
in the study of the two-component system \eqref{eq:2CHr-1}. The difficulty
lies in the fact that \eqref{eq:2CHr-1} is not symmetric in $u$
and $v$. It is of higher order in $u$ than in $v$. In \cite{BHPglobal},
the authors introduced the quantity $2^{-1}\sum_{j=0}^{m}j!^{-2}e^{2j\sigma}\|u^{(j)}\|_{2}^{2}$
in order to prove global analyticity of the unknown $u$. The system
\eqref{eq:2CHr-1} has two unknowns $u$ and $v$ and we are tempted
to introduce a sum of two quantities of this form. As it turns out,
the asymmetry of \eqref{eq:2CHr-1} messes up estimates involving
such a sum. We can overcome this difficulty by introducing the asymmetric
quantity $2^{-1}\sum_{j=1}^{m+1}j!^{-2}e^{2(j-1)\sigma}\|u^{(j)}\|_{2}^{2}+2^{-1}\sum_{j=0}^{m}j!^{-2}e^{2j\sigma}\|v^{(j)}\|_{2}^{2}$
and other relavant sums.

Although our argument is more complicated than that in \cite{BHPglobal}
because of the asymmetry, the former is simpler than the latter in
another respect. We have streamlined the argument by using fewer function
spaces. In \cite{BHPglobal}, the authors employed the spaces $G^{\delta,\theta}(\mathbb{R})$
and $E_{\delta,m}(\mathbb{R})$ in addition to $A(r)$. The two spaces
$G^{\delta,\theta}(\mathbb{R})$ and $E_{\delta,m}(\mathbb{R})$ have
similar properties, but the latter is better. In the present paper,
we make extensive use of a generalization of $E_{\delta,m}(\mathbb{R})$
and do without $G^{\delta,\theta}(\mathbb{R})$. As is stated above,
we formulate our main result in terms of $A(r)$: if the initial value
is in $\oplus^{2}A(r_{0})$ for some $r_{0}$, then the solution is
$A(r)$ for some other $r$. Such a formulation is also possible for
the main result in \cite{BHPglobal}, in which the the initial value
is assumed to be in $\oplus{}^{2}G^{1,\theta+2}(\mathbb{R})$ but
the solution is found in the larger space $A(r)$. 

To the best of our knowledge, the present work is the first result
about global analyticity for a Camassa-Holm-like system. Notice that
a global Gevrey regularity result involving a higher-order inertia
operator $(1-\partial_{x}^{2})^{s},s>1$ can be found in \cite{He-Yin},
in which the authors studied the system 
\begin{equation}
\begin{cases}
 & m_{t}-\alpha u_{x}+\beta(2mu_{x}+um_{x})+3(1-\beta)uu_{x}+\rho\rho_{x}=0,\\
 & \rho_{t}+(\rho u)_{x}=0,\\
 & m=(1-\partial_{x}^{2})^{s}u,s>1
\end{cases}\label{eq:HeYin}
\end{equation}
and proved Gevrey regularity in $x$ for any fixed $t$. Notice that
 $(1-\partial_{x}^{2})^{-s}$ has a stronger smoothing effect than
$(1-\partial_{x}^{2})^{-1}$. 

The rest of the paper is organized as follows. In Section \ref{sec:Function-spaces},
we introduce some function spaces and prove their properties. Sections
\ref{sec:Local analyticity} and \ref{sec:Global-analyticity} are
devoted to the local and global theories respectively. In the latter
we need a lot of inequalities and their proofs are given in Sections
\ref{sec:estimates1} and \ref{sec:estimates2}. 

\section{Function spaces\label{sec:Function-spaces}}

In the present paper $L^{2}=L^{2}(\mathbb{R})$, $\mathcal{C}^{\infty}(\mathbb{R})$
and their subspaces consist of \emph{real-valued} functions on $\mathbb{R}$.
We shall make frequent use of $H^{s}=H^{s}(\mathbb{R})\subset L^{2}$
with the norm $\|f\|_{s}=\|(1+\xi^{2})^{s/2}\hat{f}(\xi)\|_{L^{2}}$
and the inner product $\langle\cdot,\cdot\rangle_{s}$. We recall
some known facts about $H^{s}$ (\cite{BHPglobal,Kato-Ponce}).
\begin{lem}
\label{lem:known facts about H^s}Set $\Lambda=(1-\partial_{x}^{2})^{1/2}$.\\
(i) $\|f\|_{2}^{2}=\|\Lambda^{2}f\|_{0}^{2}=\|f\|_{0}^{2}+2\|f'\|_{0}^{2}+\|f''\|_{0}^{2}$.
\\
(ii) For $f\in L^{2}$ and $g\in H^{s}\,(s>1/2)$, we have 
\[
\|fg\|_{0}\le d(s)\|f\|_{0}\|g\|_{s},
\]
where $d(s)=\left[\int_{\mathbb{R}}(1+\xi^{2})^{-s}\,ds\right]^{1/2}$.
In particular, we have
\begin{align*}
 & \|fg\|_{0}\le\sqrt{\pi}\|f\|_{0}\|g\|_{1}\le2\|f\|_{0}\|g\|_{1},\\
 & |\langle f,gh\rangle_{0}|\le2\|f\|_{0}\|g\|_{0}\|h\|_{1}.
\end{align*}
(iii) For $f\in H^{s}\,(s\ge0)$, we have $\|\Lambda^{-2}f\|_{s+2}=\|f\|_{s}$.\\
(iv) For $f\in H^{s+1}\,(s\ge0)$, we have $\|\partial_{x}f\|_{s}\le\|f\|_{s+1}$.\\
(v) For $f,g\in H^{s}\,(s\ge1)$, we have 
\[
\|fg\|_{s}\le c_{s}(\|f\|_{s}\|g\|_{1}+\|f\|_{1}\|g\|_{s})
\]
for some constant $c_{s}>0$. In particular, for $f,g\in H^{s}\,(s\ge3)$,
we have
\[
\|fg\|_{2}\le8(\|f\|_{2}\|g\|_{1}+\|f\|_{1}\|g\|_{2}).
\]
(vi) For $f,g\in H^{s}\,(s>1/2)$, we have
\[
\|fg\|_{s}\le c(s)\|f\|_{s}\|g\|_{s},
\]
where $c(s)=\left[(1+2^{2s})\int_{\mathbb{R}}(1+\xi^{2})^{-s}\,ds\right]^{1/2}$.
In particular, we have
\begin{align*}
 & \|fg\|_{1}\le4\|f\|_{1}\|g\|_{1},\;\|fg\|_{2}\le8\|f\|_{2}\|g\|_{2},
\end{align*}
\end{lem}

For $r>0$, set 
\begin{align*}
S(r)= & \left\{ x+iy\in\mathbb{C};\,|y|<r\right\} ,\\
A(r)= & \left\{ f\colon\mathbb{R}\to\mathbb{R};\,f(z)\text{ can be analytically continued to }S(r)\right\} \\
 & \cap\left\{ f\in L_{x,y}^{2}(S(r'))\text{ for all }0<r'<r\right\} .
\end{align*}
Notice that $A(r)$ is a subspace of $L^{2}$, the space of \textit{real-valued}
square-integrable functions on $\mathbb{R}$. 

Following \cite{KatoMasuda}, we set 
\[
\|f\|_{\sigma,s}^{2}=\sum_{j=0}^{\infty}\frac{e^{2j\sigma}}{j!^{2}}\|f^{(j)}\|_{s}^{2},\;s\ge0,
\]
for $f\in H^{\infty}=\cap_{s\ge0}H^{s},$ where $f^{(j)}=\partial_{x}^{j}f(x),\,\partial_{x}=d/dx$. 
\begin{lem}
\label{lem:frechet}(\cite{KatoMasuda}) The norms $\|\cdot\|_{\sigma,s}$
have the following properties.

(i) Assume $s,s'\ge0$ and $\sigma'<\sigma$. Then there exists a
positive constant $c>0$ such that 
\[
\|f\|_{\sigma',s'}\le c\|f\|_{\sigma,s}
\]
for any $f$. 

(ii) If $f\in A(r)$ and $\sigma<\log r,s\ge0$, then $\|f\|_{\sigma,s}<\infty$.

(iii) Let $s\ge0$ be fixed. If $f\in H^{\infty}$ satisfies $\|f\|_{\sigma,s}<\infty$
for some $s\ge0$ and any $\sigma$ with $\sigma<\log r$, then $f\in A(r)$. 
\end{lem}

\begin{prop}
\label{prop:families of norms}(\cite{KatoMasuda}) When $r>0$ is
fixed, the following four families of norms determine the same topology
of $A(r)$ as a Fr\'echet space. With this topology, $A(r)$ is continuously
embedded in $H^{\infty}$.

(i) the $L_{x,y}^{2}(S(r'))$ norms $(0<r'<r)$\hspace{2em}(ii) $\|\cdot\|_{\sigma,s}$
$(s\ge0,\sigma<\log r)$

(iii) $\|\cdot\|_{\sigma,2}$ $(\sigma<\log r)$ \hspace{2em}(iv)
$\|\cdot\|_{\sigma,0}$ $(\sigma<\log r)$
\end{prop}

\begin{lem}
\label{lem:KatoMasudaLem2.4}(\cite{KatoMasuda}) Let $f_{n}\in A(r)$
be a sequence with $\|f_{n}\|_{\sigma,2}$ bounded, where $\sigma<\log r$.
If $f_{n}\to0$ in $H^{\infty}$ as $n\to\infty$, then $\|f_{n}\|_{\sigma',2}\to0$
for each $\sigma'<\sigma$. 
\end{lem}

Following \cite{BHPglobal} (with some generalization and a modified
notation), we introduce 
\[
\|f\|_{(\delta,s)}=\sup_{k\ge0}\frac{\delta^{k}(k+1)^{2}\|f^{(k)}\|_{s}}{k!}\;(0<\delta\le1,s\ge2).
\]
Do not confuse $\|\cdot\|_{(\delta,s)}$ with $\|\cdot\|_{\sigma,s}.$
Moreover, notice that a different system of notation was employed
in \cite{yamane muCH}. We introduce the Banach space $E_{\delta,s}$
by 
\[
E_{\delta,s}=\left\{ f\in\mathcal{C}^{\infty}(\mathbb{R});\,\|f\|_{(\delta,s)}<\infty\right\} .
\]
The property of $E_{\delta,s}$ is similar to that of $G^{\delta,s}$
used in \cite{BHP,BHPpower,BHPglobal} . We do without $G^{\delta,s}$
in the present paper, since $E_{\delta,s}$ is better because of the
following proposition. 
\begin{prop}
\label{prop:embed}(\cite[Lemma 5.1]{BHPglobal}) $E_{\delta,s}$
is continuously embedded in $A(\delta)$. Conversely, if $\delta<r/e$
then $A(r)$ is continuously embedded in $E_{\delta,s}.$
\end{prop}

\begin{proof}
In \cite{BHPglobal}, this proposition is stated and proved only in
the case $s=2m,\,m\in\mathbb{Z}_{+}$. The same proof is valid in
the general case.
\end{proof}
\begin{prop}
\label{prop:E estimates}

(i) If $0<\delta'<\delta\le1$ and $2\le s'<s$, then
\[
\|u\|_{(\delta',s)}\le\|u\|_{(\delta,s)},\;\|u\|_{(\delta,s')}\le\|u\|_{(\delta,s)}.
\]

(ii) If $0<\delta'<\delta\le1$, $s\ge2$, then 
\[
\|uv\|_{(\delta,s)}\le C_{s}\|u\|_{(\delta,s)}\|v\|_{(\delta,s)},\;C_{s}=18c_{s}.
\]

(iii) If $0<\delta'<\delta\le1$, we have 
\begin{align*}
 & \|\partial_{x}u\|_{(\delta',s)}\le\frac{1}{\delta-\delta'}\|u\|_{(\delta,s)},\\
 & \|\partial_{x}u\|_{(\delta,s)}\le\|u\|_{(\delta,s+1)},\\
 & \|\Lambda^{-2}\partial_{x}^{p}u\|_{(\delta,s)}\le\|u\|_{(\delta,s)}\;(p=0,1,2),\\
 & \|\Lambda^{-2}\partial_{x}u\|_{(\delta',s)}\le\frac{\|u\|_{(\delta,s)}}{\delta-\delta'}.
\end{align*}

(iv) $\|\Lambda^{-2}u\|_{(\delta,s+2)}=\|u\|_{(\delta,s)}\;(p=0,1,2)$.

(v) $\|\Lambda^{-2}\partial_{x}u\|_{(\delta',s+1)}\le\|\Lambda^{-2}\partial_{x}u\|_{(\delta',s+2)}\le\dfrac{1}{\delta-\delta'}\|u\|_{(\delta,s)}$.
\end{prop}

\begin{proof}
The estimates in (i) are trivial. We prove (ii) and (iii) following
\cite{HM}. By Lemma \ref{lem:known facts about H^s} (v), we have
\begin{align*}
\|\partial_{x}^{k}(uv)\|_{s} & \le\sum_{\ell=0}^{k}\binom{k}{\ell}\|u^{(k-\ell)}v^{(\ell)}\|_{s}\\
 & \le c_{s}\sum_{\ell=0}^{k}\binom{k}{\ell}\left(\|u^{(k-\ell)}\|_{s}\|v^{(\ell)}\|_{1}+\|u^{(k-\ell)}\|_{1}\|v^{(\ell)}\|_{s}\right)\\
 & =c_{s}\sum_{\ell=0}^{k}\binom{k}{\ell}\left(\|u^{(k-\ell)}\|_{s}\|v^{(\ell)}\|_{1}+\|u^{(\ell)}\|_{1}\|v\|_{s}^{(k-\ell)}\right).
\end{align*}
Since $s\ge2$, we have $\|v\|_{1}\le\|v\|_{s}\le\|v\|_{(\delta,s)}$
and $\|v^{(\ell)}\|_{1}\le\|v^{(\ell-1)}\|_{s}(\ell\ge1)$. Hence
\begin{align}
\|\partial_{x}^{k}(uv)\|_{s}\le & c_{s}\left(\|u^{(k)}\|_{s}\|v\|_{(\delta,s)}+\sum_{\ell=1}^{k}\binom{k}{\ell}\|u^{(k-\ell)}\|_{s}\|v^{(\ell-1)}\|_{s}\right)\label{eq:10061711}\\
 & +c_{s}\left(\|u\|_{(\delta,s)}\|v^{(k)}\|_{s}+\sum_{\ell=1}^{k}\binom{k}{\ell}\|u^{(\ell-1)}\|_{s}\|v^{(k-\ell)}\|_{s}\right).\nonumber 
\end{align}
What we want to prove is that $\delta^{k}(k+1)^{2}\|\partial_{x}^{k}(uv)\|_{s}/k!$
is bounded by $18c_{s}$ times $\|u\|_{(\delta,s)}\|v\|_{(\delta,s)}$.
In view of the symmetry in the right-hand side of \eqref{eq:10061711},
it is enough to prove \\
(a) $(k!)^{-1}\delta^{k}(k+1)^{2}\|u^{(k)}\|_{s}\|v\|_{(\delta,s)}$
is bounded by $\|u\|_{(\delta,s)}\|v\|_{(\delta,s)}$,\\
(b) $(k!)^{-1}\delta^{k}(k+1)^{2}\sum_{\ell=1}^{k}\binom{k}{\ell}\|u^{(k-\ell)}\|_{s}\|v^{(\ell-1)}\|_{s}$
is bounded by $8\|u\|_{(\delta,s)}\|v\|_{(\delta,s)}$.

The estimate (a) is trivial. Now we show (b). We have
\begin{align*}
 & \frac{\delta^{k}(k+1)^{2}}{k!}\sum_{\ell=1}^{k}\binom{k}{\ell}\|u^{(k-\ell)}\|_{s}\|v^{(\ell-1)}\|_{s}\\
 & =\frac{\delta(k+1)^{2}}{k!}\sum_{\ell=1}^{k}\binom{k}{\ell}\frac{(k-\ell)!(\ell-1)!}{(k-\ell+1)^{2}\ell^{2}}\frac{\delta^{k-\ell}(k-\ell+1)^{2}\|u^{(k-\ell)}\|_{s}}{(k-\ell)!}\frac{\delta^{\ell-1}\ell^{2}\|v^{(\ell-1)}\|_{s}}{(\ell-1)!}\\
 & \le\frac{(k+1)^{2}}{k!}\sum_{\ell=1}^{k}\binom{k}{\ell}\frac{(k-\ell)!(\ell-1)!}{(k-\ell+1)^{2}\ell^{2}}\|u\|_{(\delta,s)}\|v\|_{(\delta,s)}\\
 & \le8\|u\|_{(\delta,s)}\|v\|_{(\delta,s)}.
\end{align*}
Here we have used the fact $\sum_{\ell=1}^{k}(k-\ell+1)^{-2}\ell^{-3}(k+1)^{2}\le8$
(\cite[Lem 2.2]{HM}). The proof of (ii) is complete.

Next we prove (iii). The first inequality follows from
\begin{align*}
\|\partial_{x}u\|_{(\delta',s)} & \le\sup_{k\ge0}\frac{\delta'^{k}(k+1)^{2}\|u^{(k+1)}\|_{s}}{k!}\\
 & =\sup_{k\ge0}\frac{\delta^{k+1}(k+2)^{2}\|u^{(k+1)}\|_{s}}{(k+1)!}\frac{\delta'^{k}}{\delta^{k+1}}\frac{(k+1)^{3}}{(k+2)^{2}}\\
 & \le\|u\|_{(\delta,s)}\sup_{k\ge0}\frac{\delta'^{k}}{\delta^{k+1}}\frac{(k+1)^{3}}{(k+2)^{2}}\le\frac{\|u\|_{(\delta,s)}}{\delta-\delta'}.
\end{align*}
Here we employed $\delta'^{k}\delta^{-(k+1)}(k+1)^{3}(k+2)^{-2}\le1/(\delta-\delta')$
(\cite[(2.7)]{HM}).

The second and third inequalities of (iii) follow from $\|u^{(k+1)}\|_{s}\le\|u^{(k)}\|_{s+1}$
and $\|\Lambda^{-2}\partial_{x}^{p}u\|_{s}\le\|u\|_{s}$ respectively.
The fourth one follows from the first and the third.

The equality (iv) follows from $\|\Lambda^{-2}u\|_{s+2}=\|u\|_{s}$
and (v) follows from (iii) and (iv).
\end{proof}

\section{Local analyticity \label{sec:Local analyticity}}

\subsection{General theory}

We recall some basic facts about the autonomous Ovsyannikov theorem
following \cite{BHP,BHPpower}. Let $\left\{ X_{\delta},\|\cdot\|_{\delta}\right\} _{0<\delta\le1}$
be a (decreasing) scale of Banach spaces, i.e. each $X_{\delta}$
is a Banach space and $X_{\delta}\subset X_{\delta'},\|\cdot\|_{\delta'}\le\|\cdot\|_{\delta}$
for any $0<\delta'<\delta\le1$. For example, $\left\{ E_{\delta,s},\|\cdot\|_{\delta,s}\right\} _{0<\delta\le1}$
is a scale of Banach spaces. Assume that $F\colon X_{\delta}\to X_{\delta'}$
is a mapping satisfying the following conditions.

(a) For any $U_{0}\in X_{1}$ and $R>0$, there exist $L=L(U_{0},R)>0,M=M(U_{0},R)>0$
such that we have 
\begin{align}
\|F(U_{0})\|_{\delta} & \le\frac{M}{1-\delta}\label{eq:Ov Fu0}
\end{align}
if $0<\delta<1$ and 
\begin{equation}
\|F(U)-F(V)\|_{\delta'}\le\frac{L}{\delta-\delta'}\|U-V\|_{\delta}\label{eq:Ov Lipschitz}
\end{equation}
if $0<\delta'<\delta\le1$ and $U,V\in X_{\delta}$ satisfies $\|U-U_{0}\|_{\delta}<R,\|V-U_{0}\|_{\delta}<R$.

(b) If $U(t)$ is a holomorphic function with values in $X_{\delta}$
on the disk $D(0,a(1-\delta))=\left\{ t\in\mathbb{C}\colon|t|<a(1-\delta)\right\} $
for $a>0,0<\delta<1$ satisfying $\sup_{|t|<a(1-\delta)}\|U(t)-U_{0}\|_{\delta}<R$,
then the composite function $F(U(t))$ is a holomorphic function on
$D(0,a(1-\delta))$ with values in $X_{\delta'}$ for any $0<\delta'<\delta$.
The following autonomous Ovsyannikov theorem will be used in the next
section.
\begin{thm}
(\cite{BHP}) \label{thm:Ov}Assume that the mapping $F$ satisfies
the conditions (a) and (b). For any $U_{0}\in X_{1}$ and $R>0$,
set 
\begin{equation}
T=\frac{R}{16LR+8M}.\label{eq:T}
\end{equation}
Then, for any $\delta\in]0,1[$, the Cauchy problem 
\begin{equation}
\frac{dU}{dt}=F(U),\;U(0)=U_{0}\label{eq:Ov CP}
\end{equation}
has a unique holomorphic solution $U(t)$ in the disk $D(0,T(1-\delta))$
with values in $X_{\delta}$ satisfying
\[
\sup_{|t|<T(1-\delta)}\|U(t)-U_{0}\|_{\delta}<R.
\]
\end{thm}

\subsection{Local analyticity for the Camassa-Holm system}

We consider the analytic Cauchy problem for the generalized CH system
\eqref{eq:2CHr-1}, namely,
\begin{equation}
\begin{cases}
 & u_{t}+\beta uu_{x}+\Lambda^{-2}\partial_{x}\left[-\alpha u+\dfrac{3-\beta}{2}u^{2}+\dfrac{\beta}{2}u_{x}^{2}+v+\dfrac{1}{2}v^{2}\right]=0,\\
 & v_{t}+u_{x}+(uv)_{x}=0,\\
 & u(0,x)=u_{0}(x),\\
 & v(0,x)=v_{0}(x).
\end{cases}\label{eq:2CH CP}
\end{equation}

\begin{thm}
\label{thm:CPlocal}Let $s\ge2$. If $u_{0},v_{0}\in E_{1,s+1}$,
then there exists a positive time $T=T(u_{0},v_{0};s)$ such that
for every $\delta\in]0,1[$, the Cauchy problem \eqref{eq:2CH CP}
has a unique solution which is a holomorphic function valued in $\oplus^{2}E_{\delta,s+1}$
in the disk $D(0,T(1-\delta))$. Furthermore, the analytic lifespan
$T$ satisfies 
\[
T\approx\frac{\mathrm{const.}}{\|(u_{0},v_{0})\|_{(1,s+1)}}
\]
 for large initial values and 
\[
T\approx\mathrm{const.}
\]
for small initial values. 
\end{thm}

\begin{proof}
For $(u,v)\in\oplus^{2}E_{\delta,s+1}$, set $\|(u,v)\|_{(\delta,s+1)}=\|u\|_{(\delta,s+1)}+\|v\|_{(\delta,s+1)}$.
We employ the same notation for the norms on $E_{\delta,s+1}$ and
$\oplus^{2}E_{\delta,s+1}$. Notice that $\left\{ \oplus^{2}E_{\delta,s}\right\} _{0<\delta\le1}$
is another scale of Banach spaces. Set
\begin{align}
F_{1}(u,v) & =-\beta uu_{x}-\Lambda^{-2}\partial_{x}\left[-\alpha u+\frac{3-\beta}{2}u^{2}+\frac{\beta}{2}u_{x}^{2}+v+\frac{1}{2}v^{2}\right],\label{eq:F1}\\
F_{2}(u,v) & =-u_{x}-(uv)_{x}.\label{eq:F2}
\end{align}
We want to estimate $F_{j}(u,v)-F_{j}(u',v')$ and $F_{j}(u_{0},v_{0})$
by using Proposition \ref{prop:E estimates} when $u_{0},v_{0}\in E_{1,s+1}$
and 
\begin{equation}
\|(u,v)-(u_{0},v_{0})\|_{(\delta,s+1)}<R,\;\|(u',v')-(u_{0},v_{0})\|_{(\delta,s+1)}<R.\label{eq:930-1}
\end{equation}
First we consider $F_{2}$. By using
\begin{equation}
\|u\|_{(\delta,s+1)}\le\|(u,v)\|_{(\delta,s+1)}\le\|(u_{0},v_{0})\|_{(1,s+1)}+R\label{eq:triangle-R}
\end{equation}
and similar estimates on $u',v$ and $v'$, we get
\begin{align*}
\|uv-u'v'\|_{(\delta,s+1)} & \le\|(u-u')v\|_{(\delta,s+1)}+\|u'(v-v')\|_{(\delta,s+1)}\\
 & \le C_{s+1}\|u-u'\|_{(\delta,s+1)}\|v\|_{(\delta,s+1)}+C_{s+1}\|u'\|_{(\delta,s+1)}\|v-v'\|_{(\delta,s+1)}\\
 & \le C_{s+1}\left[\|(u_{0},v_{0})\|_{(1,s+1)}+R\right]\|(u',v')-(u_{0},v_{0})\|_{(\delta,s+1)}.
\end{align*}
Therefore
\begin{align}
 & \|(uv)_{x}-(u'v')_{x}\|_{(\delta',s+1)}\label{eq:930-2}\\
 & \le\frac{C_{s+1}\left[\|(u_{0},v_{0})\|_{(1,s+1)}+R\right]}{\delta-\delta'}\|(u,v)-(u',v')\|_{(\delta,s+1)}.\nonumber 
\end{align}
 the other hand, we have
\begin{equation}
\|u_{x}-u'_{x}\|_{(\delta',s+1)}\le\frac{\|u-u'\|_{(\delta,s+1)}}{\delta-\delta'}.\label{eq:930-3}
\end{equation}
By using \eqref{eq:930-2} and \eqref{eq:930-3}, we obtain
\begin{align}
 & \|F_{2}(u,v)-F_{2}(u',v')\|_{(\delta',s+1)}\label{eq:F2estimate}\\
 & \le\frac{C_{s+1}\left[\|(u_{0},v_{0})\|_{(1,s+1)}+R\right]+1}{\delta-\delta'}\|(u,v)-(u',v')\|_{(\delta,s+1)}.\nonumber 
\end{align}
Next, we consider $F_{1}$. By \eqref{eq:triangle-R}, we get 
\begin{align}
\|u+u'\|_{(\delta,s+1)} & \le2\left[\|(u_{0},v_{0})\|_{(1,s+1)}+R\right],\label{eq:triangle-2R}\\
\|u^{2}-u'^{2}\|_{(\delta,s+1)} & \le C_{s+1}\|u+u'\|_{(\delta,s+1)}\|u-u'\|_{(\delta,s+1)}\label{eq:930-4}\\
 & \le2C_{s+1}\left[\|u_{0}\|_{(1,s+1)}+R\right]\|u-u'\|_{(\delta,s+1)}.\nonumber 
\end{align}
Since $2uu_{x}-2u'u'_{x}=(u^{2}-u'^{2})_{x}$, we get 
\begin{equation}
\|uu_{x}-u'u'_{x}\|_{(\delta',s+1)}\le\frac{C_{s+1}\left[\|u_{0}\|_{(1,s+1)}+R\right]}{\delta-\delta'}\|u-u'\|_{(\delta,s+1)}.\label{eq:930-5}
\end{equation}
On the other hand, Proposition \ref{prop:E estimates} (iii) implies
\begin{align}
\|\Lambda^{-2}\partial_{x}(u-u')\|_{(\delta',s+1)} & \le\frac{1}{\delta-\delta'}\|u-u'\|_{(\delta,s+1)},\label{eq:930-6}\\
\|\Lambda^{-2}\partial_{x}(v-v')\|_{(\delta',s+1)} & \le\frac{1}{\delta-\delta'}\|v-v'\|_{(\delta,s+1)}.\label{eq:1016-2}
\end{align}
Similarly, the estimate \eqref{eq:930-4} implies
\begin{align}
\|\Lambda^{-2}\partial_{x}(u^{2}-u'^{2})\|_{(\delta',s+1)} & \le\frac{1}{\delta-\delta'}\|u^{2}-u'^{2}\|_{(\delta,s+1)}\label{eq:930-7}\\
 & \le\frac{2C_{s+1}\left[\|u_{0}\|_{(1,s+1)}+R\right]}{\delta-\delta'}\|u-u'\|_{(\delta,s+1)}\nonumber 
\end{align}
and 
\begin{align}
\|\Lambda^{-2}\partial_{x}(v^{2}-v'^{2})\|_{(\delta',s+1)} & \le\frac{2C_{s+1}\left[\|v_{0}\|_{(1,s+1)}+R\right]}{\delta-\delta'}\|v-v'\|_{(\delta,s+1)}.\label{eq:1016-3}
\end{align}
What remains is the estimate on $\Lambda^{-2}\partial_{x}(u_{x}^{2}-u_{x}'^{2})$.

By using \eqref{eq:triangle-2R}, Proposition \ref{prop:E estimates}
(v) and 
\begin{align*}
\|u_{x}^{2}-u_{x}'^{2}\|_{(\delta,s)} & \le C_{s}\|(u+u')_{x}\|_{(\delta,s)}\|(u-u')_{x}\|_{(\delta,s)}\\
 & \le C_{s}\|u+u'\|_{(\delta,s+1)}\|u-u'\|_{(\delta,s+1)},
\end{align*}
 we obtain
\begin{equation}
\|\Lambda^{-2}\partial_{x}(u_{x}^{2}-u_{x}'^{2})\|_{(\delta',s+1)}\le\frac{2C_{s}\left[\|u_{0}\|_{(1,s+1)}+R\right]}{\delta-\delta'}\|u-u'\|_{(\delta,s+1)}.\label{eq:1016-1}
\end{equation}
Combining \eqref{eq:930-6}, \eqref{eq:1016-2}, \eqref{eq:930-7},
\eqref{eq:1016-3} and \eqref{eq:1016-1}, we obtain
\begin{align}
 & \|F_{1}(u,v)-F_{1}(u',v')\|_{(\delta',s+1)}\label{eq:F1 estimate}\\
 & \le\frac{C'_{s}(2|\beta|+|3-\beta|+1)\left[\|(u_{0},v_{0})\|_{(1,s+1)}+R\right]+|\alpha|+1}{\delta-\delta'}\nonumber \\
 & \quad\times\|(u,v)-(u',v')\|_{(\delta,s+1)},\nonumber 
\end{align}
where $C_{s}'=\max[C_{s},C_{s+1}]$. The two estimates \eqref{eq:F1 estimate}
and \eqref{eq:F2estimate} mean that we have 
\begin{equation}
\|(F_{1},F_{2})(u,v)-(F_{1},F_{2})(u',v')\|_{(\delta',s)}\le\frac{L}{\delta-\delta'}\|(u,v)-(u',v')\|_{(\delta,s)},\label{eq:L ineq}
\end{equation}
where
\begin{equation}
L=C_{s}'\left(2|\beta|+|3-\beta|+2\right)\left[\|(u_{0},v_{0})\|_{(1,s+1)}+R\right]+|\alpha|+2,\label{eq:L def}
\end{equation}
if $\|(u,v)-(u_{0},v_{0})\|_{(\delta,s+1)}<R,\;\|(u',v')-(u_{0},v_{0})\|_{(\delta,s+1)}<R$.

Next, we evaluate what corresponds to the constant $M$ in the general
theory. We have
\begin{align*}
\|u_{0}(u_{0})_{x}\|_{(\delta,s+1)} & \le\frac{(C_{s+1}/2)\|u_{0}\|_{(1,s+1)}^{2}}{1-\delta},\allowdisplaybreaks\\
\|\Lambda^{-2}\partial_{x}u_{0}\|_{(\delta,s+1)} & \le\frac{\|u_{0}\|_{(1,s+1)}}{1-\delta},\allowdisplaybreaks\\
\|\Lambda^{-2}\partial_{x}u_{0}^{2}\|_{(\delta,s+1)} & \le\frac{\|u_{0}^{2}\|_{(1,s+1)}}{1-\delta}\le\frac{C_{s+1}\|u_{0}\|_{(1,s+1)}^{2}}{1-\delta},\allowdisplaybreaks\\
\|\Lambda^{-2}\partial_{x}(\partial_{x}u_{0})^{2}\|_{(\delta,s+1)} & \le\frac{\|(\partial_{x}u_{0})^{2}\|_{(1,s)}}{1-\delta}\le\frac{C_{s+1}\|u_{0}\|_{(1,s+1)}^{2}}{1-\delta},\allowdisplaybreaks\\
\|\Lambda^{-2}\partial_{x}v_{0}\|_{(\delta,s+1)} & \le\frac{\|\Lambda^{-2}v_{0}\|_{(1,s+1)}}{1-\delta}\le\frac{\|v_{0}\|_{(1,s+1)}}{1-\delta},\allowdisplaybreaks\\
\|\Lambda^{-2}\partial_{x}v_{0}^{2}\|_{(\delta,s+1)} & \le\frac{\|v_{0}^{2}\|_{(1,s+1)}}{1-\delta}\le\frac{C_{s+1}\|v_{0}\|_{(1,s+1)}^{2}}{1-\delta}.\allowdisplaybreaks
\end{align*}
These inequalities lead to 
\begin{align}
\|F_{1}(u_{0},v_{0})\|_{(\delta,s+1)}\le & \frac{M_{1}}{1-\delta},\label{eq:0113-1}\\
M_{1}= & \frac{C_{s+1}}{2}(2|\beta|+|3-\beta|)\|u_{0}\|_{(1,s+1)}^{2}+|\alpha|\|u_{0}\|_{(1,s+1)}\nonumber \\
 & +\|v_{0}\|_{(1,s+1)}+\frac{C_{s+1}}{2}\|v_{0}\|_{(1,s+1)}^{2}.\nonumber 
\end{align}
On the other hand, we have 
\begin{align*}
\|(u_{0})_{x}\|_{(\delta,s+1)} & \le\frac{\|u_{0}\|_{(1,s+1)}}{1-\delta},\allowdisplaybreaks\\
\|(u_{0}v_{0})_{x}\|_{(\delta,s+1)} & \le\frac{\|u_{0}v_{0}\|_{(1,s+1)}}{1-\delta}\le\frac{C_{s+1}\|u_{0}\|_{(1,s+1)}\|v_{0}\|_{(1,s+1)}}{1-\delta}\allowdisplaybreaks\\
 & \le\frac{(C_{s+1}/2)(\|u_{0}\|_{(1,s+1)}+\|v_{0}\|_{(1,s+1)})^{2}}{1-\delta}\allowdisplaybreaks\\
 & =\frac{(C_{s+1}/2)\|(u_{0},v_{0})\|_{(1,s+1)}^{2}}{1-\delta}.
\end{align*}
Hence 
\begin{align}
 & \|F_{2}(u_{0},v_{0})\|_{(\delta,s+1)}\le\frac{M_{2}}{1-\delta},\label{eq:0113-2}\\
 & M_{2}=\|u_{0}\|_{1,s+1}+\frac{C_{s+1}}{2}\|(u_{0},v_{0})\|_{1,s+1}^{2}.\nonumber 
\end{align}
Finally by combining \eqref{eq:0113-1} and \eqref{eq:0113-2} , we
obtain
\begin{align}
 & \|(F_{1},F_{2})(u_{0},v_{0})\|_{(\delta,s+1)}\le\frac{M}{1-\delta},\label{eq:M ineq}\\
 & M=\frac{C_{s+1}}{2}(2|\beta|+|3-\beta|+2)\|(u_{0},v_{0})\|_{1,s+1}^{2}+(|\alpha|+2)\|(u_{0},v_{0})\|_{1,s+1}.\label{eq:M def}
\end{align}
We can apply the general theory with $L,M$ as in \eqref{eq:L def}
and \eqref{eq:M def}. We set $R=\|(u_{0},v_{0})\|_{1,s+1}$. Then
\[
T=\frac{R}{16LR+8M}=\frac{R}{(\gamma_{1}R+\gamma_{2})R}=\frac{1}{\gamma_{1}R+\gamma_{2}}
\]
for some positive constants $\gamma_{1},\gamma_{2}$ depending only
on $\alpha,\beta,C_{s},C_{s+1}$. We have $T\to1/\gamma_{2}\ne0$
as $\|(u_{0},v_{0})\|_{1,s+1}=R\to0$ and $T\approx1/(\gamma_{1}R)$
as $\|(u_{0},v_{0})\|_{1,s+1}=R\to\infty$.
\end{proof}
In Theorem \ref{thm:CPlocal}, we assumed the initial values $u_{0}$
and $v_{0}$ were in $E_{1,s+1}$. We can relax this assumption as
in the following theorem. 
\begin{thm}
\label{thm:CP local2}Let $0<\Delta\le1,s\ge2$. If $(u_{0},v_{0})\in\oplus^{2}E_{\Delta,s+1}$,
then there exists $T_{\Delta}>0$ such that the Cauchy problem \eqref{eq:2CH CP}
has a unique solution which is a holomorphic function valued in $\oplus^{2}E_{\Delta d,s+1}$
in the disk $D(0,T_{\Delta}(1-d))$ for every $d\in]0,1[$.
\end{thm}

\begin{proof}
Set $X_{d}=\oplus^{2}E_{\Delta d,s+1},\|\cdot\|_{(d,s+1)}^{(\Delta)}=\|\cdot\|_{(\Delta d,s+1)}$.
Then $\left\{ X_{d},\|\cdot\|_{(d,s+1)}^{(\Delta)}\right\} _{0<d\le1}$
is a (decreasing) scale of of Banach spaces and $(u_{0},v_{0})\in X_{1}$.
Assume $\|u-u_{0}\|_{(d,s+1)}^{(\Delta)}<R,\|v-u_{0}\|_{(d,s+1)}^{(\Delta)}<R$
and $0<d'<d\le1$. Then by setting $\delta=\Delta d,\delta'=\Delta d'$,
we obtain the following counterpart to \eqref{eq:L def}: 
\begin{equation}
\|(F_{1},F_{2})(u,v)-(F_{1},F_{2})(u',v')\|_{(d',s)}^{(\Delta)}\le\frac{L_{\Delta}}{d-d'}\|(u,v)-(u',v')\|_{(d,s)}^{(\Delta)}.\label{eq:L ineq-1}
\end{equation}
Here 
\begin{equation}
L_{\Delta}=C'_{s}\Delta^{-1}\left(2|\beta|+|3-\beta|+2\right)\left[\|(u_{0},v_{0})\|_{(1,s+1)}^{(\Delta)}+R\right]+|\alpha|+1,\label{eq:L def-1}
\end{equation}
and $\|(u,v)-(u_{0},v_{0})\|_{(d,s+1)}^{(\Delta)}<R,\;\|(u',v')-(u_{0},v_{0})\|_{(d,s+1)}^{(\Delta)}<R$.
On the other hand, we have 
\begin{align*}
 & \|(F_{1},F_{2})(u_{0},v_{0})\|_{d,s+1}^{(\Delta)}\le\frac{M_{\Delta}}{1-d},\\
 & M_{\Delta}=\frac{C_{s+1}}{2\Delta}(2|\beta|+|3-\beta|+2)\left(\|(u_{0},v_{0})\|_{1,s+1}^{(\Delta)}\right)^{2}+\frac{|\alpha|+2}{\Delta}\|(u_{0},v_{0})\|_{1,s+1}^{(\Delta)}.\qedhere
\end{align*}
\end{proof}

\section{Global analyticity\label{sec:Global-analyticity}}

\subsection{Main result}

We recall a known result about global-in-time solutions to \eqref{eq:2CH CP}
in Sobolev spaces. We have
\begin{thm}
(\cite[Theorems 3.1, 5.1]{ChenLiu})\label{thm:Chen Liu} Assume $0<\beta<2$
and let $s>3/2$. If $(u_{0},v_{0})\in H^{s}(\mathbb{\mathbb{R}})\times H^{s-1}(\mathbb{\mathbb{R}})$
and $\inf_{x\in\mathbb{R}}v_{0}(x)>-1$, then \eqref{eq:2CH CP} has
a unique solution $(u,v)$ in the space $\mathcal{C}([0,\infty),H^{s}(\mathbb{\mathbb{R}})\times H^{s-1}(\mathbb{\mathbb{R}}))\cap\mathcal{C}^{1}([0,\infty),H^{s-1}(\mathbb{\mathbb{R}})\times H^{s-2}(\mathbb{\mathbb{R}}))$. 
\end{thm}

Our main result is the following theorem, the proof of which shall
be given in Propositions \ref{prop:u analytic in x} and \ref{prop:global analyticity}
below.
\begin{thm}
\label{thm:main} Assume $0<\beta<2$. If $u_{0},v_{0}\in A(r_{0})$
for some $r_{0}>0$ and $\inf_{x\in\mathbb{R}}v_{0}(x)>-1$, then
the unique solution $(u,v)$ in Theorem \ref{thm:Chen Liu} to the
Cauchy problem \eqref{eq:2CH CP} belongs to $\oplus^{2}\mathcal{C}^{\omega}([0,\infty)_{t}\times\mathbb{R}_{x})$.
Moreover, there exists a continuous function $\sigma(t)$ such that
we have $u(\cdot,t),v(\cdot,t)\in A(e^{\sigma(t)})$ for $t\in[0,T]$.
(The function $\sigma(t)$ shall be defined in Proposition \ref{prop:u analytic in x}.)
\end{thm}

\subsection{Kato-Masuda theory}

In \cite{KatoMasuda}, the authors used their theory of Liapnov families
to prove a regularity result about the KdV and other equations. Later,
it was applied to a generalized Camassa-Holm equation in \cite{BHPglobal}.
Here we recall the abstract theorem in \cite{KatoMasuda} in a weaker,
more concrete form. In applications, making a suitable choice of the
subset $\mathcal{O}$ of $Z$ is essential.
\begin{thm}
\label{thm:KatoMasuda}Let $X$ and $Z$ be Banach spaces. Assume
that $Z$ is a dense subspace of $X$. Let $\mathcal{O}$ be an open
subset of $Z$ and $F$ be a continuous mapping from $Z$ to $X$.
Let $\left\{ \Phi_{s};-\infty<s<\bar{s}\le\infty\right\} $ be a family
of non-negative real-valued functions on $Z$ satisfying the conditions
(a), (b) and (c) below.

(a) The Fr\'echet partial derivative of $\Phi_{s}(U)$ in $U\in Z$
exists not only in $\mathcal{L}(Z;\mathbb{R})$ but also in $\mathcal{L}(X;\mathbb{R})$.
It is denoted by $D\Phi_{s}(U)$. This statement makes sense because
$\mathcal{L}(X;\mathbb{R})\subset\mathcal{L}(Z;\mathbb{R})$ by the
canonical identification. {[}(a) follows from (b) below.{]}

(b) The Fr\'echet derivative of $\Phi_{s}(U)$ in $(s,U)$ exists
not only in $\mathcal{L}(\mathbb{R}\times Z;\mathbb{R})$ but also
in $\mathcal{L}(\mathbb{R}\times X;\mathbb{R})$ and is continuous
from $\mathbb{R}\times Z$ to $\mathcal{L}(\mathbb{R}\times X;\mathbb{R})$.
This statement makes sense because $\mathcal{L}(\mathbb{R}\times X;\mathbb{R})\subset\mathcal{L}(\mathbb{R}\times Z;\mathbb{R})$
by the canonical identification.

(c) There exist positive constants $K,L$ and $M$ such that 
\begin{align}
\left|\left\langle F(U),D\Phi_{s}(U)\right\rangle \right| & \le K\Phi_{s}(U)+L\Phi_{s}(U)^{1/2}\partial_{s}\Phi_{s}(U)+M\partial_{s}\Phi_{s}(U)\label{eq:liapnov F}
\end{align}
holds for any $U\in\mathcal{O}.$ Here $\langle\cdot,\cdot\rangle$
(no subscript) is the pairing of $X$ and $\mathcal{L}(X;\mathbb{R})$.

Let $U\in\mathcal{C}([0,T];\,\mathcal{O})\cap\mathcal{C}^{1}([0,T];\,X)$
be the solution to the Cauchy problem
\begin{equation}
\begin{cases}
\dfrac{dU}{dt}=F(U),\\
U(0,x)=U_{0}(x).
\end{cases}\label{eq:KMCP}
\end{equation}
For a fixed constant $s_{0}<\bar{s}$, set 
\begin{align*}
r(t) & =\Phi_{s_{0}}(U_{0})e^{Kt},\\
s(t) & =s_{0}-\int_{0}^{t}\left(Lr(\tau)^{1/2}+M\right)\,d\tau=s_{0}-\frac{2L\Phi_{s_{0}}(U_{0})^{1/2}}{K}(e^{Kt/2}-1)-Mt
\end{align*}
for $t\in[0,T]$. Then we have 
\[
\Phi_{s(t)}\left(U(t)\right)\le r(t),\;t\in[0,T].
\]
\end{thm}

Later we will use this theorem when $X=\oplus^{2}H^{m+2},\,Z=\oplus^{2}H^{m+4}$
and $\Phi_{s}$ is related to the Sobolev norms. Roughly speaking,
Theorem \ref{thm:KatoMasuda} means that the regularity of $U(t)$
for $t\in[0,T]$ follows from that of $U_{0}$.

\subsection{Pairing and the main estimate}

For $\sigma\in\mathbb{R}$, $(u,v)\in\oplus^{2}H^{\infty}$ and a
positive integer $m$, set

\begin{align*}
\Phi_{\sigma,m}(u,v) & =\Phi_{\sigma,m}^{(1)}(u)+\Phi_{\sigma,m}^{(2)}(v),\\
\Phi_{\sigma,m}^{(1)}(u) & =\frac{1}{2}\sum_{j=1}^{m+1}\frac{e^{2(j-1)\sigma}}{j!^{2}}\|u^{(j)}\|_{2}^{2}\\
 & =\frac{1}{2}\left(\|u^{(1)}\|_{2}^{2}+\frac{e^{2\sigma}}{2!^{2}}\|u^{(2)}\|_{2}^{2}+\cdots+\frac{e^{2m\sigma}}{(m+1)!^{2}}\|u^{(m+1)}\|_{2}^{2}\right),\\
\Phi_{\sigma,m}^{(2)}(v) & =\frac{1}{2}\sum_{j=0}^{m}\frac{e^{2j\sigma}}{j!^{2}}\|v^{(j)}\|_{2}^{2}\\
 & =\frac{1}{2}\left(\|v\|_{2}^{2}+\frac{e^{2\sigma}}{1!^{2}}\|v^{(1)}\|_{2}^{2}+\cdots+\frac{e^{2m\sigma}}{m!^{2}}\|v^{(m)}\|_{2}^{2}\right).
\end{align*}
The asymmetry of the exponents and the ranges of $j$ corresponds
to that of \eqref{eq:2CHr-1}: its second equation involves $u_{x}$,
while the first does not involve $v_{x}.$ We have 
\[
\|u\|_{2}^{2}+2\lim_{m\to\infty}e^{2\sigma}\Phi_{\sigma,m}^{(1)}(u)=\|u\|_{\sigma,2}^{2},\;\lim_{m\to\infty}2\Phi_{\sigma,m}^{(2)}(v)=\|v\|_{\sigma,2}^{2}.
\]
It is trivial that $\lim_{m\to\infty}\Phi_{\sigma,m}(u,v)<\infty$
if and only if $\|u\|_{\sigma,2}<\infty,\|v\|_{\sigma,2}<\infty$.
Later we will use
\[
\partial_{\sigma}\Phi_{\sigma,m}(u,v)=\sum_{j=2}^{m+1}\frac{(j-1)e^{2(j-1)\sigma}}{j!^{2}}\|u^{(j)}\|_{2}^{2}+\sum_{j=1}^{m}\frac{je^{2j\sigma}}{j!^{2}}\|v^{(j)}\|_{2}^{2}.
\]

\begin{prop}
\label{prop:frechet}Set $F(u,v)=(F_{1}(u,v),F_{2}(u,v))$, where
$F_{1}$ and $F_{2}$ are as in \eqref{eq:F1} and \eqref{eq:F2}.
Then $F$ is a continuous mapping from $\oplus^{2}H^{m+4}$ to $\oplus^{2}H^{m+2}$
and we have
\begin{align}
 & \langle F(u,v),D\Phi_{\sigma,m}(u,v)\rangle\label{eq:frechet}\\
 & =\sum_{j=1}^{m+1}\frac{e^{2(j-1)\sigma}}{j!^{2}}\langle u^{(j)},\partial_{x}^{j}F_{1}(u,v)\rangle_{2}+\sum_{j=0}^{m}\frac{e^{2j\sigma}}{j!^{2}}\langle v^{(j)},\partial_{x}^{j}F_{2}(u,v)\rangle_{2},\nonumber 
\end{align}
where $D\Phi_{\sigma,m}$ is the Fr\'echet derivative and the bracket
on the left-hand side is the pairing of $\oplus^{2}H^{m+2}$ and its
dual $(\oplus^{2}H^{m+2})^{*}\simeq\oplus^{2}H^{m+2}$.
\end{prop}

\begin{proof}
Set
\begin{align*}
\Psi_{j}^{(1)}(u,v) & =\Psi_{j}(u)=\frac{1}{2}\|u^{(j)}\|_{2}^{2}=\frac{1}{2}\|\Lambda^{2}u^{(j)}\|_{0}^{2},\\
\Psi_{j}^{(2)}(u,v) & =\Psi_{j}(v)=\frac{1}{2}\|v^{(j)}\|_{2}^{2}=\frac{1}{2}\|\Lambda^{2}v^{(j)}\|_{0}^{2}.
\end{align*}
Then 
\[
\Phi_{\sigma,m}(u,v)=\sum_{j=1}^{m+1}\frac{e^{2(j-1)\sigma}}{j!^{2}}\Psi_{j}^{(1)}(u,v)+\sum_{j=0}^{m}\frac{e^{2j\sigma}}{j!^{2}}\Psi_{j}^{(2)}(u,v).
\]
We have
\begin{align*}
\langle(w_{1},w_{2}),D\Psi_{j}^{(1)}(u,v)\rangle & =D\Psi_{j}^{(1)}(u,v)\left((w_{1},w_{2})\right)=\left.\frac{d}{d\tau}\Psi_{j}^{(1)}\left((u,v)+\tau(w_{1},w_{2})\right)\right|_{\tau=0}\\
 & =\left.\frac{d}{d\tau}\Psi_{j}\left(u+\tau w_{1}\right)\right|_{\tau=0}=\langle u^{(j)},w_{1}^{(j)}\rangle_{2}
\end{align*}
and similarly
\[
\langle(w_{1},w_{2}),D\Psi_{j}^{(2)}(u,v)\rangle=\langle v^{(j)},w_{2}^{(j)}\rangle_{2}.
\]
The proposition follows immediately.
\end{proof}
\begin{prop}
\label{prop:main estimate}There exist positive constants $K_{1},K_{2},L_{1},L_{2},M_{1},M_{2},M_{3}$
independent of $u,v$ and $\sigma$ such that we have
\begin{align}
\left|\langle F(u,v),D\Phi_{\sigma,m}(u,v)\rangle\right|\le & \left[K_{1}+K_{2}\|(u,v)\|_{3}\right]\Phi_{\sigma,m}(u,v)\label{eq:liapnov F_mu}\\
 & +\left(L_{1}+L_{2}e^{\sigma}\right)\Phi_{\sigma,m}(u,v)^{1/2}\partial_{\sigma}\Phi_{\sigma,m}(u,v)\nonumber \\
 & +\left[M_{1}+(M_{2}+M_{3}e^{2\sigma})\|(u,v)\|_{3}\right]\partial_{\sigma}\Phi_{\sigma,m}(u,v).\nonumber 
\end{align}
for $(u,v)\in\oplus^{2}H^{m+4}$, where $D\Phi_{\sigma,m}$ is the
Fr\'echet derivative of $\Phi_{\sigma,m}$.
\end{prop}

\begin{proof}
The estimate follows from \eqref{eq:frechet} and the estimates \eqref{eq:subsection1conclusion},
\eqref{eq:subsection2conclusion}, \eqref{eq:subsection3conclusion},
\eqref{eq:subsection4conclusion}, \eqref{eq:subsection5conclusion},
\eqref{eq:subsection6conclusion}, \eqref{eq:subsection7conclusion},
\eqref{eq:subsection8conclusion} below.
\end{proof}

\subsection{Analyticity in the space variable\label{subsec:Analyticity x}}

In this subsection, we prove a part of Theorem \ref{thm:main}. We
assume that $(u_{0},v_{0})\in\oplus^{2}A(r_{0})$ is as in Theorem
\ref{thm:main}. Then we can apply Theorem \ref{thm:Chen Liu} with
arbitrarily large $s$. We will prove the analyticity of $u(t)$ and
$v(t)$ in $x$ for each fixed $t$.
\begin{prop}
\label{prop:u analytic in x}Let $(u_{0},v_{0})$ and $(u,v)$ be
as in Theorem \ref{thm:main}. Fix $\sigma_{0}<\log r_{0}$ and $T>0$.
Then there exists positive constants $K,L,M$ such that if we set
\begin{align*}
 & \Phi_{\sigma_{0},\infty}(u_{0},v_{0})=\lim_{m\to\infty}\Phi_{\sigma_{0},m}(u_{0},v_{0}),\\
 & \sigma(t)=\sigma_{0}-\frac{2L}{K}\Phi_{\sigma_{0},\infty}(u_{0},v_{0})^{1/2}(e^{Kt/2}-1)-Mt,
\end{align*}
then we have $\left(u(t),v(t)\right)\in\oplus^{2}A(e^{\sigma(t)})\subset\oplus^{2}A(e^{\sigma(T)})$
for $t\in[0,T]$.
\end{prop}

\begin{proof}
Assume $u_{0},v_{0}\in A(r_{0})$ and set $\bar{\sigma}=\log r_{0}$.
Theorem \ref{thm:Chen Liu} implies $u(t),v(t)\in H^{\infty}$. Set
\begin{align*}
\mu_{0} & =1+\max\left\{ \|(u,v)\|_{3};\,t\in[0,T]\right\} >0,\\
\mathcal{O} & =\left\{ (u,v)\in\oplus^{2}H^{m+4};\,\|(u,v)\|_{3}<\mu_{0}\right\} ,\\
K & =K_{1}+K_{2}\mu_{0},\\
L & =L_{1}+L_{2}e^{\bar{\sigma}},\\
M & =M_{1}+(M_{2}+M_{3}e^{2\bar{\sigma}})\mu_{0}.
\end{align*}
Then $\left(u(t),v(t)\right)\in\mathcal{O}$ for $t\in[0,T]$. Proposition
\ref{prop:main estimate} implies that 
\begin{align}
 & \left|\langle F(u,v),D\Phi_{\sigma,m}(u,v)\rangle\right|\label{eq:importantestimate}\\
 & \le K\Phi_{\sigma,m}(u,v)+L\Phi_{\sigma,m}(u,v)^{1/2}\partial_{\sigma}\Phi_{\sigma,m}(u,v)+M\partial_{\sigma}\Phi_{\sigma,m}(u,v)\nonumber 
\end{align}
holds for any $(u,v)\in\mathcal{O},\sigma\le\bar{\sigma}.$ This corresponds
to \eqref{eq:liapnov F} in Theorem \ref{thm:KatoMasuda}. Since $u_{0}\in A(r_{0})$,
we have $\|u_{0}\|_{\sigma_{0},2}<\infty$ by Proposition \ref{prop:families of norms}. 

Set 
\begin{align*}
\rho_{m}(t) & =\Phi_{\sigma_{0},m}(u_{0},v_{0})e^{Kt},\\
\sigma_{m}(t) & =\sigma_{0}-\int_{0}^{t}\left(L\rho_{m}(\tau)^{1/2}+M\right)\,d\tau\\
 & =\sigma_{0}-\frac{2L}{K}\Phi_{\sigma_{0},m}(u_{0},v_{0})^{1/2}(e^{Kt/2}-1)-Mt.
\end{align*}
These functions correspond to $r(t)$ and $s(t)$ in Theorem \ref{thm:KatoMasuda}.
We can apply Theorem \ref{thm:KatoMasuda} to obtain
\begin{equation}
\Phi_{\sigma_{m}(t),m}(u(t),v(t))\le\rho_{m}(t),\quad t\in[0,T].\label{eq:Liaponov conclusion}
\end{equation}
Set 
\begin{align*}
\rho(t) & =\lim_{m\to\infty}\rho_{m}(t),\\
\Phi_{\sigma_{0},\infty}(u_{0},v_{0}) & =\lim_{m\to\infty}\Phi_{\sigma_{0},m}(u_{0},v_{0})=\frac{1}{2}e^{-2\sigma_{0}}\left(\|u\|_{\sigma_{0},2}^{2}-\|u\|_{2}^{2}\right)+\frac{1}{2}\|v\|_{\sigma_{0},2}^{2}.\\
\sigma(t) & =\sigma_{0}-\frac{2L}{K}\Phi_{\sigma_{0},\infty}(u_{0},v_{0})^{1/2}(e^{Kt/2}-1)-Mt,\\
 & =\sigma_{0}-\frac{L}{K}\left[e^{-2\sigma_{0}}\left(\|u\|_{\sigma_{0},2}^{2}-\|u\|_{2}^{2}\right)+\|v\|_{\sigma_{0},2}^{2}\right]^{1/2}(e^{Kt/2}-1)-Mt.
\end{align*}
Notice that $\sigma(t)$ is a decreasing function in $t$. We have
$\rho_{m}(t)\le\rho_{m+1}(t)\le\rho(t)$ and $\sigma_{m}(t)\ge\sigma_{m+1}(t)\ge\sigma(t)$,
$\sigma_{m}(t)\to\sigma(t)\,(\text{as }m\to\infty)$. By Fatou's Lemma
and \eqref{eq:Liaponov conclusion}, we get
\begin{align*}
 & \|u(t)\|_{\sigma(t),2}^{2}+\|v(t)\|_{\sigma(t),2}^{2}\\
 & =\sum_{j=0}^{\infty}\frac{1}{j!^{2}}e^{2j\sigma(t)}\left(\|u^{(j)}(t)\|_{2}^{2}+\|v^{(j)}(t)\|_{2}^{2}\right)\\
 & \le\liminf_{m\to\infty}\sum_{j=0}^{m}\frac{1}{j!^{2}}e^{2j\sigma_{m}(t)}\left(\|u^{(j)}(t)\|_{2}^{2}+\|v^{(j)}(t)\|_{2}^{2}\right)\\
 & =\|u(t)\|_{2}^{2}+2\liminf_{m\to\infty}\left(e^{2\sigma_{m}(t)}\Phi_{\sigma_{m}(t),m}^{(1)}\left(u(t)\right)+\Phi_{\sigma_{m}(t),m}^{(2)}\left(v(t)\right)\right)\\
 & \le\|u(t)\|_{2}^{2}+2\liminf_{m\to\infty}\left(e^{2\sigma_{m}(t)}+1\right)\rho_{m}(t)\\
 & =\|u(t)\|_{2}^{2}+2\left(e^{2\sigma(t)}+1\right)\rho(t)<\infty.
\end{align*}
Therefore $u(t),v(t)\in A(e^{\sigma(t)})\subset A(e^{\sigma(T)})$
for $t\in[0,T]$.
\end{proof}
\begin{prop}
The mapping $[0,T]\to\oplus^{2}A(e^{2\pi\sigma(T)}),\,t\mapsto(u(t),v(t))$
is continuous.
\end{prop}

\begin{proof}
Let $\left\{ t_{n}\right\} \subset[0,T]$ be a sequence converging
to $t_{\infty}\in[0,T]$. We have $(u(t_{n}),v(t_{n}))\to(u(t_{\infty}),v(t_{\infty}))$
in $\oplus^{2}H^{\infty}$. In particular, $\|u(t_{n})\|_{2}$ is
bounded. On the other hand, $\|u(t_{n})\|_{\sigma(T),2},\|v(t_{n})\|_{\sigma(T),2}\,(n\ge0)$
are bounded since
\begin{align*}
\|u(t)\|_{\sigma(T),2}^{2}+\|v(t)\|_{\sigma(T),2}^{2} & \le\|u(t)\|_{\sigma(t),2}^{2}+\|v(t)\|_{\sigma(t),2}^{2}\\
 & \le\|u(t)\|_{2}^{2}+2\left(e^{2\sigma(t)}+1\right)\rho(t)
\end{align*}
by the proof of Proposition \ref{prop:u analytic in x}. Lemma \ref{lem:KatoMasudaLem2.4}
implies that $(u(t_{n}),v(t_{n}))$ converges to $(u(t_{\infty}),v(t_{\infty}))$
with respect to $\|\cdot\|_{\sigma',2}\,(\sigma'<\sigma(T))$. By
Proposition \ref{prop:families of norms}, this means convergence
in $\oplus^{2}A(e^{\sigma(T)})$.
\end{proof}

\subsection{Analyticity in the space and time variables\label{subsec:Analyticity t, x}}

We continue the proof of Theorem \ref{thm:main}. In the previous
subsection, we have established the analyticity in the space variable.
Here, we will prove the analyticity in the space and time variables.
By convention, a real-analytic function on a closed interval is real-analytic
on some open neighborhood of the closed interval.
\begin{prop}
Under the condition of Theorem \ref{thm:main}, for any $T>0$ there
exists $\delta_{T}>0$ such that we have $(u,v)\in\mathcal{C}^{\omega}([0,T],\oplus^{2}A(\delta_{T}))$.
\end{prop}

\begin{proof}
We have $(u_{0},v_{0})\in\oplus^{2}E_{\Delta,s+1}$ for any $\Delta<r_{0}/e$.
Let $s\ge2$ and assume $\Delta\le1$. By Theorem \ref{thm:CP local2},
there exists $T_{\Delta}>0$ such that the Cauchy problem \eqref{eq:2CH CP}
has a unique solution $(\tilde{u},\tilde{v})\in\mathcal{C}^{\omega}(|t|\le T_{\Delta}(1-d),\oplus^{2}E_{\Delta d,s+1})$
for $0<d<1$. We have $(\tilde{u},\tilde{v})=(u,v)$ by the local
uniqueness of \cite[Theorem 3.1]{ChenLiu}, where $(u,v)$ is the
solution in Theorem \ref{thm:Chen Liu}. Set $d=1/2,\,T'_{\Delta}=T_{\Delta}/2$.
Then $(\tilde{u},\tilde{v})=(u,v)\in\mathcal{C}^{\omega}(|t|\le T'_{\Delta},\oplus^{2}E_{\Delta/2,s+1})$.
By Proposition \ref{prop:embed}, a convergent series in $E_{\Delta/2,s+1}$
is convergent in $A(\Delta/2)$. We have $(u,v)\in\mathcal{C}^{\omega}(|t|\le T'_{\Delta},\oplus^{2}A(\Delta/2))$.

We have shown that $(u,v)$ is analytic in $t$ at least locally.
Our next step is to show that it is analytic in $t$ globally. Set
\begin{align*}
 & S=\left\{ T>0;\,(u,v)\in\mathcal{C}^{\omega}([0,T],\oplus^{2}A(\delta_{T}))\;\text{for some}\;\delta_{T}>0\right\} \ni T'_{\Delta},\\
 & T^{*}=\sup S\ge T'_{\Delta}.
\end{align*}
We prove $T^{*}=\infty$ by contradiction. Assume $T^{*}<\infty$.
By Propositions \ref{prop:u analytic in x} and \ref{prop:embed},
there exists $\delta^{*}\in]0,\Delta/2]$ such that 
\[
(u(T^{*}),v(T^{*}))\in\oplus^{2}E_{\delta',s+1}\,(0<\delta'<\delta^{*}).
\]
By Theorem \ref{thm:CP local2} (with $t$ replaced with $t-T^{*}$),
there exists $\varepsilon>0$ and $(\hat{u},\hat{v})\in\mathcal{C}^{\omega}([T^{*}-\varepsilon,T^{*}+\varepsilon],\oplus^{2}E_{\delta'/2,s+1})$
such that 
\begin{align*}
 & \hat{u}_{t}+\beta\hat{u}\hat{u}_{x}+(1-\partial_{x}^{2})^{-1}\partial_{x}\left[-\alpha\hat{u}+\dfrac{3-\beta}{2}\hat{u}^{2}+\dfrac{\beta}{2}\hat{u}_{x}^{2}+\hat{v}+\dfrac{1}{2}\hat{v}^{2}\right]=0,\\
 & \hat{v}_{t}+\hat{u}_{x}+(\hat{u}\hat{v})_{x}=0,\\
 & (\hat{u},\hat{v})|_{t=T^{*}}=(u(T^{*}),v(T^{*})).
\end{align*}
By the local uniqueness, we have $(\hat{u},\hat{v})=(u,v)$. Namely,
$(\hat{u},\hat{v})$ is an extension of $(u,v)$ up to $t\le T^{*}+\varepsilon$
(valued in $\oplus^{2}E_{\delta'/2,s+1}\subset\oplus^{2}A(\delta'/2)$).
Therefore $T^{*}+\varepsilon\in S$. This is a contradiction.

Finally we prove the analyticity in $(t,x)$.
\end{proof}
\begin{prop}
\label{prop:global analyticity}Under the condition of Theorem \ref{thm:main},
the Cauchy problem \eqref{eq:2CH CP} has a unique solution $(u,v)\in\oplus^{2}\mathcal{C}^{\omega}([0,\infty)\times\mathbb{R})$.
\end{prop}

\begin{proof}
Let $T$ be fixed. For $r>0$ sufficiently small, we have 
\[
\partial_{t}^{k}u(t)=\frac{k!}{2\pi i}\int_{|\tau-t|=r}\frac{u(\tau)}{(\tau-t)^{k+1}}\,d\tau,\;t\in[0,T].
\]
The integral is performed in $A(\delta_{T})$ and converges with respect
to $\|\cdot\|_{(\sigma,0)}$ ($\sigma<\log\delta_{T}$). By Cauchy's
estimate, there exists $C_{0}>1/r$ such that 
\begin{align*}
\|\partial_{t}^{k}u(t)\|_{(\sigma,0)} & \le C_{0}k!r^{-k}<C_{0}^{k+1}k!.
\end{align*}
Therefore we have 
\begin{align*}
\|\partial_{x}^{j}\partial_{t}^{k}u(\cdot,t)\|_{0} & \le C_{0}^{k+1}e^{-j\sigma}j!k!
\end{align*}
and there exists $C>0$ such that 
\begin{equation}
\left\Vert \partial_{x}^{j}\partial_{t}^{k}u\right\Vert _{L^{2}(\mathbb{R}\times[0,T])}\le\sqrt{T}C^{j+k+1}(j+k)!.\label{eq:Komatsu}
\end{equation}
 Set $\Delta=\partial^{2}/\partial x^{2}+\partial^{2}/\partial t^{2}$.
The binomial expansion of $\Delta^{\ell}\,(\ell=0,1,2,\dots)$ and
\eqref{eq:Komatsu} yield
\begin{align*}
\left\Vert \Delta^{\ell}u\right\Vert _{L^{2}(\mathbb{R}\times[0,T])} & \le\sqrt{T}C^{2\ell+1}(2\ell)!\sum_{p=0}^{\ell}\binom{\ell}{p}\\
 & \le\sqrt{T}C^{2\ell+1}(2\ell)!2^{\ell}.
\end{align*}
This estimates implies the real-analyticity of $u$ due to \cite{Komatsu}
or \cite{Kotake}.
\end{proof}

\section{Proof of Proposition \ref{prop:main estimate}: Part 1\label{sec:estimates1}}

The goal of this section is to get an estimate of the first sum of
the right-hand side of \eqref{eq:frechet} in Proposition \ref{prop:frechet}.
We have to estimate the six quantities
\begin{align*}
 & \langle u^{(j)},\partial_{x}^{j}(uu_{x})\rangle_{2},\langle u^{(j)},\Lambda^{-2}\partial_{x}^{j+1}u\rangle_{2},\langle u^{(j)},\Lambda^{-2}\partial_{x}^{j+1}u^{2}\rangle_{2},\\
 & \langle u^{(j)},\Lambda^{-2}\partial_{x}^{j+1}u_{x}^{2}\rangle_{2},\langle u^{(j)},\Lambda^{-2}\partial_{x}^{j+1}v\rangle_{2},\langle u^{(j)},\Lambda^{-2}\partial_{x}^{j+1}v^{2}\rangle_{2}
\end{align*}
for $0<j=1,2,\dots,m+1$. 

\subsection{Estimate 1\label{subsec:Estimates1} }

We have $\langle u^{(j)},\partial_{x}^{j}(uu_{x})\rangle_{2}=P_{j}+Q_{j}$,
where 
\begin{align*}
P_{j} & =\langle u^{(j)},uu^{(j+1)}\rangle_{2},\\
Q_{j} & =\sum_{\ell=1}^{j}\binom{j}{\ell}\langle u^{(j)},u^{(\ell)}u^{(j-\ell+1)}\rangle_{2}\,(j\ge1).
\end{align*}
By \cite[(6.5)]{BHPglobal}, we have 
\[
|P_{j}|\le2^{4}\|u\|_{2}\|u^{(j)}\|_{2}^{2}.
\]
Therefore 
\begin{equation}
\left|\sum_{j=1}^{m+1}\frac{e^{2(j-1)\sigma}}{j!^{2}}P_{j}\right|\le2^{5}\|u\|_{2}\Phi_{\sigma,m}^{(1)}(u)\le2^{5}\|u\|_{2}\Phi_{\sigma,m}(u,v).\label{eq:P_j sum}
\end{equation}
On the other hand, by using Lemma \ref{lem:known facts about H^s}
(v) and splitting the sum into the terms for $\ell=1,j$ and those
in-between, we get 
\begin{align}
|Q_{j}|\le & 8\sum_{\ell=1}^{j}\binom{j}{\ell}\|u^{(j)}\|_{2}\left(\|u^{(\ell)}\|_{2}\|u^{(j-\ell+1)}\|_{1}+\|u^{(\ell)}\|_{1}\|u^{(j-\ell+1)}\|_{2}\right)\label{eq:Qj}\\
\le & 16(j+1)\|u^{(1)}\|_{2}\|u^{(j)}\|_{2}^{2}+q_{j},\nonumber \\
q_{j}=8 & \sum_{\ell=2}^{j-1}\binom{j}{\ell}\|u^{(j)}\|_{2}\left(\|u^{(\ell)}\|_{2}\|u^{(j-\ell)}\|_{2}+\|u^{(\ell-1)}\|_{2}\|u^{(j-\ell+1)}\|_{2}\right).\label{eq:qj}
\end{align}
We set $q_{j}=0$ if $j\le2$. As for the first term in the right-hand
side of \eqref{eq:Qj}, we have $16(j+1)=32+16(j-1)$ and 
\begin{align}
 & \sum_{j=1}^{m+1}\frac{e^{2(j-1)\sigma}}{j!^{2}}\cdot16(j+1)\|u^{(1)}\|_{2}\|u^{(j)}\|_{2}^{2}\label{eq:1010-1}\\
 & \le64\|u^{(1)}\|_{2}\Phi_{\sigma,m}^{(1)}(u)+16\|u^{(1)}\|_{2}\partial_{\sigma}\Phi_{\sigma,m}^{(1)}(u)\nonumber \\
 & \le64\|u\|_{3}\Phi_{\sigma,m}(u,v)+16\|u\|_{3}\partial_{\sigma}\Phi_{\sigma,m}(u,v).\nonumber 
\end{align}
Set $a_{k}=e^{k\sigma}\|u^{(k)}\|_{2}^{2}/k!$. Then, \eqref{eq:qj},
\eqref{eq:AB1}, \eqref{eq:AB2} and Proposition \ref{prop:A and tilde A}
imply
\begin{align}
 & \sum_{j=3}^{m+1}\frac{e^{2(j-1)\sigma}}{j!^{2}}q_{j}\label{eq:1010-2}\\
 & \le8e^{-2\sigma}\sum_{j=3}^{m+1}\sum_{\ell=2}^{j-1}a_{j}a_{\ell}a_{j-\ell}+8e^{-2\sigma}\sum_{j=3}^{m+1}\sum_{\ell=2}^{j-1}\frac{j-\ell+1}{\ell}a_{j}a_{\ell-1}a_{j-\ell+1}\nonumber \\
 & \le\frac{32\pi}{\sqrt{3}}e^{\sigma}\sqrt{\Phi_{\sigma,m}^{(1)}(u)}\,\partial_{\sigma}\Phi_{\sigma,m}^{(1)}(u)\le\frac{32\pi}{\sqrt{3}}e^{\sigma}\sqrt{\Phi_{\sigma,m}(u,v)}\,\partial_{\sigma}\Phi_{\sigma,m}(u,v).\nonumber 
\end{align}
By combining \eqref{eq:P_j sum}, \eqref{eq:Qj}, \eqref{eq:1010-1}
and \eqref{eq:1010-2}, we get 
\begin{align}
 & \left|\sum_{j=1}^{m+1}\frac{e^{2(j-1)\sigma}}{j!^{2}}\langle u^{(j)},\partial_{x}^{j}(uu_{x})\rangle_{2}\right|\label{eq:subsection1conclusion}\\
 & \le96\|u\|_{3}\Phi_{\sigma,m}(u,v)+\left(16\|u\|_{3}+\frac{32\pi}{\sqrt{3}}e^{\sigma}\sqrt{\Phi_{\sigma,m}(u,v)}\right)\partial_{\sigma}\Phi_{\sigma,m}(u,v).\nonumber 
\end{align}

\subsection{Estimate 2}

Since 
\begin{align*}
 & |\langle u^{(j)},\Lambda^{-2}\partial_{x}^{j+1}u\rangle_{2}|\le\|u^{(j)}\|_{2}\|\Lambda^{-2}\partial_{x}^{j+1}u\|_{2}\\
 & \le\|u^{(j)}\|_{2}\|\partial_{x}^{j+1}u\|_{0}\le\|u^{(j)}\|_{2}^{2},
\end{align*}
we have 
\begin{align}
 & \left|\sum_{j=1}^{m+1}\frac{e^{2(j-1)\sigma}}{j!^{2}}\langle u^{(j)},\partial_{x}^{j+1}\Lambda^{-2}u\rangle_{2}\right|\label{eq:subsection2conclusion}\\
 & \le\sum_{j=1}^{m+1}\frac{e^{2(j-1)\sigma}}{j!^{2}}\|u^{(j)}\|_{2}^{2}\le2\Phi_{\sigma,m}^{(1)}(u)\le2\Phi_{\sigma,m}(u,v).\nonumber 
\end{align}

\subsection{Estimate 3}

In Subsection \ref{subsec:Estimates1}, we considered $\langle u^{(j)},\partial_{x}^{j}(uu_{x})\rangle_{2}$,
which is $1/2$ times $\langle u^{(j)},\partial_{x}^{j+1}u^{2}\rangle_{2}$.
In the present subsection we consider $\langle u^{(j)},\partial_{x}^{j+1}\Lambda^{-2}u^{2}\rangle_{2}$.
We just neglect $\Lambda^{-2}\colon H^{2}\to H^{2}$ and follow \eqref{eq:subsection1conclusion}
to obtain 
\begin{align}
 & \left|\sum_{j=1}^{m+1}\frac{e^{2(j-1)\sigma}}{j!^{2}}\langle u^{(j)},\Lambda^{-2}\partial_{x}^{j+1}u^{2}\rangle_{2}\right|\label{eq:subsection3conclusion}\\
 & \le192\|u\|_{3}\Phi_{\sigma,m}(u,v)+\left(32\|u\|_{3}+\frac{64\pi}{\sqrt{3}}e^{\sigma}\sqrt{\Phi_{\sigma,m}(u,v)}\right)\partial_{\sigma}\Phi_{\sigma,m}(u,v).\nonumber 
\end{align}

\subsection{Estimate 4}

We have
\begin{align*}
\langle u^{(j)},\Lambda^{-2}\partial_{x}^{j+1}(u_{x}^{2})\rangle_{2} & =\langle u^{(j)},(\Lambda^{-2}\partial_{x}^{2})\partial_{x}^{j-1}(u_{x}^{2})\rangle_{2}\\
 & =\sum_{k=0}^{j-1}\binom{j-1}{k}\langle u^{(j)},(\Lambda^{-2}\partial_{x}^{2})(u^{(k+1)}u^{(j-k)})\rangle_{2}\\
 & =\sum_{\ell=1}^{j}\binom{j-1}{\ell-1}\langle u^{(j)},(\Lambda^{-2}\partial_{x}^{2})(u^{(\ell)}u^{(j-\ell+1)})\rangle_{2}.
\end{align*}
Since the norm of $\Lambda^{-2}\partial_{x}^{2}\colon H^{2}\to H^{2}$
is 1, this operator can be neglected in estimating $\langle u^{(j)},(\Lambda^{-2}\partial_{x}^{2})(u^{(\ell)}u^{(j-\ell+1)})\rangle_{2}$.
Moreover we have
\[
\binom{j-1}{\ell-1}\le\binom{j}{\ell}\,(j\ge1).
\]
Therefore, following the calculation about $Q_{j}$ in Subsection
\ref{subsec:Estimates1}, we get 
\begin{align}
 & \left|\sum_{j=1}^{m+1}\frac{e^{2j\sigma}}{j!^{2}}\langle u^{(j)},\Lambda^{-2}\partial_{x}^{j+1}(u_{x}^{2})\rangle_{2}\right|\label{eq:subsection4conclusion}\\
 & \le64\|u\|_{3}\Phi_{\sigma,m}(u,v)+\left(16\|u\|_{3}+\frac{32\pi}{\sqrt{3}}e^{\sigma}\sqrt{\Phi_{\sigma,m}(u,v)}\right)\partial_{\sigma}\Phi_{\sigma,m}(u,v).\nonumber 
\end{align}

\subsection{Estimate 5\label{subsec:Estimates5}}

Since
\begin{align*}
 & \left|\langle u^{(j)},\Lambda^{-2}\partial_{x}^{j+1}v\rangle_{2}\right|\le\|u^{(j)}\|_{2}\|\Lambda^{-2}\partial_{x}v^{(j)}\|_{2}\\
 & \le\|u^{(j)}\|_{2}\|v^{(j)}\|_{1}\le\|u^{(j)}\|_{2}\|v^{(j-1)}\|_{2}\le\frac{1}{2}\left(\|u^{(j)}\|_{2}^{2}+\|v^{(j-1)}\|_{2}^{2}\right),
\end{align*}
we have
\begin{align}
 & \left|\sum_{j=1}^{m+1}\frac{e^{2(j-1)\sigma}}{j!^{2}}\langle u^{(j)},\Lambda^{-2}\partial_{x}^{j+1}v\rangle_{2}\right|\le\frac{1}{2}\sum_{j=1}^{m+1}\frac{e^{2(j-1)\sigma}}{j!^{2}}\left(\|u^{(j)}\|_{2}^{2}+\|v^{(j-1)}\|_{2}^{2}\right)\label{eq:subsection5conclusion}\\
 & \le\frac{1}{2}\sum_{j=1}^{m+1}\frac{e^{2(j-1)\sigma}}{j!^{2}}\|u^{(j)}\|_{2}^{2}+\frac{1}{2}\sum_{k=0}^{m}\frac{e^{2k\sigma}}{k!^{2}}\|v^{(k)}\|_{2}^{2}=\Phi_{\sigma,m}(u,v).\nonumber 
\end{align}

\subsection{Estimate 6}

We consider $\langle u^{(j)},\Lambda^{-2}\partial_{x}^{j+1}v^{2}\rangle_{2}$.
When $j=1$, we have 
\begin{align*}
 & \left|\langle u_{x},\Lambda^{-2}\partial_{x}^{2}v^{2}\rangle_{2}\right|\le\|u_{x}\|_{2}\|\Lambda^{-2}\partial_{x}^{2}v^{2}\|_{2}\le\|u\|_{3}\|v^{2}\|_{2}\le8\|u\|_{3}\|v\|_{2}^{2}.
\end{align*}
We multiply it with $e^{2(1-1)\sigma}/1!^{2}=1$. The product is bounded
by 16$\|u\|_{3}\Phi_{\sigma,m}^{(2)}(v)$.

Next assume $j\ge2$. Since the norm of $\Lambda^{-2}\partial_{x}^{2}\colon H^{2}\to H^{2}$
is 1, we have 
\begin{align*}
 & \left|\langle u^{(j)},\Lambda^{-2}\partial_{x}^{j+1}v^{2}\rangle_{2}\right|\le||u^{(j)}\|_{2}\|\partial_{x}^{j-1}v^{2}\|_{2}\\
 & \le8||u^{(j)}\|_{2}\sum_{\ell=0}^{j-1}\binom{j-1}{\ell}\|v^{(\ell)}\|_{2}\|v^{(j-\ell-1)}\|_{2}.
\end{align*}
Separating the term for $\ell=0$ from the other terms, we obtain
\begin{align*}
 & \left|\sum_{j=2}^{m+1}\frac{e^{2(j-1)\sigma}}{j!^{2}}\langle u^{(j)},\Lambda^{-2}\partial_{x}^{j+1}v^{2}\rangle_{2}\right|\le S_{0}+S_{1},\\
 & S_{0}=8\|v\|_{2}\sum_{j=2}^{m+1}\frac{e^{2(j-1)\sigma}}{j!^{2}}||u^{(j)}\|_{2}\|v^{(j-1)}\|_{2},\\
 & S_{1}=8\sum_{j=2}^{m+1}\frac{e^{2(j-1)\sigma}}{j!^{2}}||u^{(j)}\|_{2}\sum_{\ell=1}^{j-1}\binom{j-1}{\ell}\|v^{(\ell)}\|_{2}\|v^{(j-\ell-1)}\|_{2}.
\end{align*}
We have
\begin{align*}
S_{0} & \le4\|v\|_{2}\sum_{j=2}^{m+1}\frac{e^{2(j-1)\sigma}}{j!^{2}}\left(||u^{(j)}\|_{2}^{2}+\|v^{(j-1)}\|_{2}^{2}\right)\\
 & \le4\|v\|_{2}\sum_{j=2}^{m+1}\frac{e^{2(j-1)\sigma}}{j!^{2}}||u^{(j)}\|_{2}^{2}+4\|v\|_{2}\sum_{k=1}^{m}\frac{e^{2k\sigma}}{k!^{2}}\|v^{(k)}\|_{2}^{2}\\
 & \le8\|v\|_{2}\Phi_{\sigma,m}^{(1)}(u)+8\|v\|_{2}\Phi_{\sigma,m}^{(2)}(v)=8\|v\|_{2}\Phi_{\sigma,m}(u,v).
\end{align*}
On the other hand, \eqref{eq:AB3} and Proposition \ref{prop:A and tilde A}
imply
\begin{align*}
S_{1} & =8e^{-\sigma}\sum_{j=2}^{m+1}\sum_{\ell=1}^{j-1}\frac{1}{j}a_{j}b_{\ell}b_{j-\ell-1}\\
 & \le\frac{16\pi}{\sqrt{6}}\sqrt{\Phi_{\sigma,m}^{(2)}(v)}\,\sqrt{\partial_{\sigma}\Phi_{\sigma,m}^{(1)}(u)\cdot\partial_{\sigma}\Phi_{\sigma,m}^{(2)}(v)}\\
 & \le\frac{16\pi}{\sqrt{6}}\sqrt{\Phi_{\sigma,m}(u,v)}\,\partial_{\sigma}\Phi_{\sigma,m}(u,v).
\end{align*}
Summing up, we obtain
\begin{align}
 & \left|\sum_{j=1}^{m+1}\frac{e^{2(j-1)\sigma}}{j!^{2}}\langle u^{(j)},\Lambda^{-2}\partial_{x}^{j+1}v^{2}\rangle_{2}\right|\label{eq:subsection6conclusion}\\
 & \le(16\|u\|_{3}+8\|v\|_{2})\Phi_{\sigma,m}(u,v)+\frac{16\pi}{\sqrt{6}}\sqrt{\Phi_{\sigma,m}(u,v)}\partial_{\sigma}\Phi_{\sigma,m}(u,v).\nonumber 
\end{align}

\section{Proof of Proposition \ref{prop:main estimate}: Part 2\label{sec:estimates2}}

The goal of this section is to get an estimate of the second sum of
the right-hand side of \eqref{eq:frechet} in Proposition \ref{prop:frechet}.
We have to estimate the two quantities 
\[
\langle v^{(j)},u^{(j+1)}\rangle_{2},\langle v^{(j)},(uv)^{(j+1)}\rangle_{2}
\]
for $j=0,1,\dots,m$. Notice that the range of $j$ is different from
that in the previous section.

\subsection{Estimate 7\label{subsec:Estimates7}}

We have 
\[
\langle v^{(j)},u^{(j+1)}\rangle_{2}=I_{0}+2I_{1}+I_{2},\quad I_{i}=\langle v^{(j+i)},u^{(j+i+1)}\rangle_{0},
\]
where $\langle\cdot,\cdot\rangle_{0}$ is the inner product of $H^{0}=L^{2}$.
For $i=0,1,2$, we have
\[
|I_{i}|\le\|u^{(j+i+1)}\|_{0}\|v^{(j+i)}\|_{0}\le\|u^{(j+1)}\|_{2}\|v^{(j)}\|_{2}.
\]
By using $2XY\le X^{2}/(j+1)+(j+1)Y^{2}$, we get
\begin{equation}
|I_{i}|\le\frac{1}{2}\left[\frac{1}{j+1}\|u^{(j+1)}\|_{2}^{2}+(j+1)\|v^{(j)}\|_{2}^{2}\right].\label{eq:Ii}
\end{equation}
for $i=0,1,2$. Therefore
\begin{equation}
\left|\sum_{j=0}^{m}\frac{e^{2j\sigma}}{j!^{2}}\langle v^{(j)},u^{(j+1)}\rangle_{2}\right|\le2\sum_{j=0}^{m}\frac{e^{2j\sigma}}{j!^{2}}\left[\frac{1}{j+1}\|u^{(j+1)}\|_{2}^{2}+(1+j)\|v^{(j)}\|_{2}^{2}\right].\label{eq:v j u j+1}
\end{equation}
Here the sum involving $u$ is 
\begin{align*}
 & \sum_{j=0}^{m}\frac{e^{2j\sigma}}{j!^{2}}\frac{1}{j+1}\|u^{(j+1)}\|_{2}^{2}=\sum_{j=0}^{m}\frac{e^{2j\sigma}}{(j+1)!^{2}}\|u^{(j+1)}\|_{2}^{2}+\sum_{j=0}^{m}\frac{e^{2j\sigma}j}{(j+1)!^{2}}\|u^{(j+1)}\|_{2}^{2}\\
 & =\sum_{k=1}^{m+1}\frac{e^{2(k-1)\sigma}}{k!^{2}}\|u^{(k)}\|_{2}^{2}+\sum_{k=1}^{m+1}\frac{e^{2(k-1)\sigma}(k-1)}{k!^{2}}\|u^{(k)}\|_{2}^{2}=2\Phi_{\sigma,m}^{(1)}(u)+\partial_{\sigma}\Phi_{\sigma,m}^{(1)}(u).
\end{align*}
The sum involving $v$ is 
\[
\sum_{j=0}^{m}\frac{e^{2j\sigma}}{j!^{2}}(1+j)\|v^{(j)}\|_{2}^{2}=2\Phi_{\sigma,m}^{(2)}(v)+\partial_{\sigma}\Phi_{\sigma,m}^{(2)}(v).
\]
 Hence we get 
\begin{align}
 & \left|\sum_{j=0}^{m}\frac{e^{2j\sigma}}{j!^{2}}\langle v^{(j)},u^{(j+1)}\rangle_{2}\right|\label{eq:subsection7conclusion}\\
 & \le4\Phi_{\sigma,m}^{(1)}(u)+2\partial_{\sigma}\Phi_{\sigma,m}^{(1)}(u)+4\Phi_{\sigma,m}^{(2)}(v)+2\partial_{\sigma}\Phi_{\sigma,m}^{(2)}(v)\nonumber \\
 & =4\Phi_{\sigma,m}(u,v)+2\partial_{\sigma}\Phi_{\sigma,m}(u,v).\nonumber 
\end{align}

\subsection{Estimate 8\label{subsec:Estimates8}}

We shall prove

\begin{align}
 & \left|\sum_{j=1}^{m}\frac{e^{2j\sigma}}{j!^{2}}\langle v^{(j)},(uv)^{(j+1)}\rangle_{2}\right|\label{eq:subsection8conclusion}\\
 & \le\left[(66+16e^{2\sigma})\|u\|_{3}+(18+8e^{2\sigma})\|v\|_{3}\right]\Phi_{\sigma,m}(u,v)\allowdisplaybreaks\nonumber \\
 & \quad+\frac{16\pi}{\sqrt{3}}(1+\sqrt{2})e^{\sigma}\sqrt{\Phi_{\sigma,m}(u,v)}\partial_{\sigma}\Phi_{\sigma,m}(u,v)\allowdisplaybreaks\nonumber \\
 & \quad+\left[8\|u\|_{3}+(4e^{2\sigma}+13)\|v\|_{3}\right]\partial_{\sigma}\Phi_{\sigma,m}(u,v).\nonumber 
\end{align}
This inequality is the result of \eqref{eq:1sttermconclusion}, \eqref{eq:2ndtermconclusion}
and \eqref{eq:3rdtermconclusion} below. 

Since $(uv)^{(j+1)}=(uv_{x}+u_{x}v)^{(j)}$, we have $\langle v^{(j)},(uv)^{(j+1)}\rangle_{2}=J_{1}+J_{2}+J_{3}$,
where
\begin{align*}
 & J_{1}=\langle v^{(j)},u^{(j+1)}v\rangle_{2},\\
 & J_{2}=\langle v^{(j)},uv^{(j+1)}\rangle_{2},\\
 & J_{3}=\sum_{\ell=1}^{j}\binom{j}{\ell}\langle v^{(j)},u^{(\ell)}v^{(j-\ell+1)}+u^{(j-\ell+1)}v^{(\ell)}\rangle_{2}.
\end{align*}
Notice that $J_{3}=0$ if $j=0$.

\subsubsection{First term}

We decompose $J_{1}$ into a sum of $L^{2}$ inner products by $J_{1}=J_{10}+2J_{11}+J_{12}$,
where $J_{1i}=\langle v^{(j+i)},\partial_{x}^{i}(u^{(j+1)}v)\rangle_{0}$.
We have 
\begin{align}
|J_{10}| & \le\|v^{(j)}\|_{0}\|u^{(j+1)}v\|_{0}\label{eq:J10}\\
 & \le\|v^{(j)}\|_{0}\cdot2\|v\|_{1}\|u^{(j+1)}\|_{0}\le\|v\|_{1}\cdot2\|u^{(j+1)}\|_{2}\|v^{(j)}\|_{2}\nonumber \\
 & \le\|v\|_{1}\left[\frac{1}{j+1}\|u^{(j+1)}\|_{2}^{2}+(1+j)\|v^{(j)}\|_{2}^{2}\right].\nonumber 
\end{align}
The right-hand side is similar to that of \eqref{eq:Ii}. As for $J_{11}$,
we have 
\begin{align}
|J_{11}| & \le\|v^{(j+1)}\|_{0}\left(\|u^{(j+1)}v_{x}\|_{0}+\|u^{(j+2)}v\|_{0}\right)\label{eq:J11}\\
 & \le\|v^{(j+1)}\|_{0}\left(2\|u^{(j+1)}\|_{0}\|v_{x}\|_{1}+2\|u^{(j+2)}\|_{0}\|v\|_{1}\right)\nonumber \\
 & \le4\|v^{(j)}\|_{2}\|u^{(j+1)}\|_{2}\|v\|_{2}\nonumber \\
 & \le2\|v\|_{2}\left[\frac{1}{j+1}\|u^{(j+1)}\|_{2}^{2}+(1+j)\|v^{(j)}\|_{2}^{2}\right].\nonumber 
\end{align}
We further decompose $J_{12}=\langle v^{(j+2)},\partial_{x}^{2}(u^{(j+1)}v)\rangle_{0}$
by $J_{12}=J_{120}+2J_{121}+J_{122}$, where 
\begin{align*}
 & J_{12i}=\langle v^{(j+2)},u^{(j+3-i)}v^{(i)}\rangle_{0}.
\end{align*}
We have 
\begin{align}
|J_{12i}| & \le\|v^{(j+2)}\|_{0}\|u^{(j+3-i)}v^{(i)}\|_{0}\le2\|v^{(j+2)}\|_{0}\|u^{(j+3-i)}\|_{0}\|v^{(i)}\|_{1}\nonumber \\
 & =2\|v^{(j)}\|_{2}\|u^{(j+1)}\|_{2}\|v\|_{3}\nonumber \\
 & \le\|v\|_{3}\left[\frac{1}{j+1}\|u^{(j+1)}\|_{2}^{2}+(1+j)\|v^{(j)}\|_{2}^{2}\right],\nonumber \\
|J_{12}| & \le4\|v\|_{3}\left[\frac{1}{j+1}\|u^{(j+1)}\|_{2}^{2}+(1+j)\|v^{(j)}\|_{2}^{2}\right].\label{eq:J12}
\end{align}
By combining \eqref{eq:J10}, \eqref{eq:J11} and \eqref{eq:J12},
we obtain
\[
|J_{1}|\le9\|v\|_{3}\left[\frac{1}{j+1}\|u^{(j+1)}\|_{2}^{2}+(1+j)\|v^{(j)}\|_{2}^{2}\right].
\]
Following \eqref{eq:v j u j+1} and \eqref{eq:subsection7conclusion},
we get
\begin{equation}
\left|\sum_{j=0}^{m}\frac{e^{2j\sigma}}{j!^{2}}J_{1}\right|\le9\|v\|_{3}\left[2\Phi_{\sigma,m}(u,v)+\partial_{\sigma}\Phi_{\sigma,m}(u,v)\right].\label{eq:1sttermconclusion}
\end{equation}

\subsubsection{Second term}

We decompose $J_{2}$ into a sum of $L^{2}$ inner products by $J_{2}=J_{20}+2J_{21}+J_{22}$,
where $J_{2i}=\langle v^{(j+i)},\partial_{x}^{i}(uv^{(j+1)})\rangle_{0}$.
We have
\begin{align*}
|J_{20}| & \le\|v^{(j)}\|_{0}\|uv^{(j+1)}\|_{0}\le2\|v^{(j)}\|_{0}\|u\|_{1}\|v^{(j+1)}\|_{0}\\
 & \le2\|u\|_{1}\|v^{(j)}\|_{2}^{2},\\
|J_{21}| & \le\|v^{(j+1)}\|_{0}\left(\|u_{x}v^{(j+1)}\|_{0}+\|uv^{(j+2)}\|_{0}\right)\\
 & \le\|v^{(j+1)}\|_{0}\left(2\|u_{x}\|_{1}\|v^{(j+1)}\|_{0}+2\|u\|_{1}\|v^{(j+2)}\|_{0}\right)\le4\|u\|_{2}\|v^{(j)}\|_{2}^{2}.
\end{align*}
Therefore
\begin{align*}
\left|\sum_{j=0}^{m}\frac{e^{2j\sigma}}{j!^{2}}J_{20}\right| & \le4\|u\|_{1}\Phi_{\sigma,m}^{(2)}(v)\le4\|u\|_{1}\Phi_{\sigma,m}(u,v),\\
\left|\sum_{j=0}^{m}\frac{e^{2j\sigma}}{j!^{2}}J_{21}\right| & \le8\|u\|_{2}\Phi_{\sigma,m}^{(2)}(v)\le8\|u\|_{2}\Phi_{\sigma,m}(u,v).
\end{align*}
We further decompose $J_{22}=\langle v^{(j+2)},\partial_{x}^{2}(uv^{(j+1)})\rangle_{0}$
by $J_{22}=J_{220}+2J_{221}+J_{222}$, where 
\[
J_{22i}=\langle v^{(j+2)},u^{(2-i)}v^{(j+i+1)}\rangle_{0}.
\]
For $i=0,1$, we have 
\begin{align*}
|J_{22i}| & \le\|v^{(j+2)}\|_{0}\|u^{(2-i)}v^{(j+i+1)}\|_{0}\le2\|v^{(j+2)}\|_{0}\|u^{(2-i)}\|_{1}\|v^{(j+i+1)}\|_{0}\\
 & \le2\|u\|_{3}\|v^{(j)}\|_{2}^{2}.
\end{align*}
We get an alternative expression of $J_{222}=\langle v^{(j+2)},uv^{(j+3)}\rangle_{0}$
by integration by parts. We have
\begin{align*}
J_{222} & =\int_{\mathbb{R}}\frac{1}{2}u\partial_{x}(v^{(j+2)})^{2}\,dx=-\frac{1}{2}\int_{\mathbb{R}}u_{x}(v^{(j+2)})^{2}\,dx=-\frac{1}{2}\langle u_{x}v^{(j+2)},v^{(j+2)}\rangle_{0}.
\end{align*}
Therefore
\begin{align*}
|J_{222}| & \le\frac{1}{2}\|u_{x}v^{(j+2)}\|_{0}\|v^{(j+2)}\|_{0}\le\|u_{x}\|_{1}\|v^{(j+2)}\|_{0}^{2}\le\|u\|_{2}\|v^{(j)}\|_{2}^{2}.
\end{align*}
Summing up, we get
\begin{align*}
\left|\sum_{j=0}^{m}\frac{e^{2j\sigma}}{j!^{2}}J_{22}\right| & \le\sum_{j=0}^{m}\frac{e^{2j\sigma}}{j!^{2}}\left(|J_{220}|+2|J_{221}|+|J_{222}|\right)\\
 & \le14\|u\|_{3}\Phi_{\sigma,m}^{(2)}(v)\le14\|u\|_{3}\Phi_{\sigma,m}(u,v).
\end{align*}
Finally we obtain
\begin{equation}
\left|\sum_{j=0}^{m}\frac{e^{2j\sigma}}{j!^{2}}J_{2}\right|\le34\|u\|_{3}\Phi_{\sigma,m}(u,v).\label{eq:2ndtermconclusion}
\end{equation}

\subsubsection{Third term}

Recall 
\[
J_{3}=\sum_{\ell=1}^{j}\binom{j}{\ell}\langle v^{(j)},u^{(\ell)}v^{(j-\ell+1)}+u^{(j-\ell+1)}v^{(\ell)}\rangle_{2}\;(1\le j\le m)
\]
and that $J_{3}=0$ if $j=0$. We assume $j\ge1$. We have
\begin{align*}
 & J_{3}=K_{j}+L_{j},\,K_{j}=K_{j1}+K_{j2},\,L_{j}=L_{j1}+L_{j2},\\
 & K_{j1}=j\langle v^{(j)},u_{x}v^{(j)}\rangle_{2},\;K_{j2}=\sum_{\ell=2}^{j}\binom{j}{\ell}\langle v^{(j)},u^{(\ell)}v^{(j-\ell+1)}\rangle_{2},\\
 & L_{j1}=j\langle v^{(j)},u^{(j)}v_{x}\rangle_{2},\;L_{j2}=\sum_{\ell=2}^{j}\binom{j}{\ell}\langle v^{(j)},u^{(j-\ell+1)}v^{(\ell)}\rangle_{2}.
\end{align*}
If $j=1$, then $K_{j2}=L_{j2}=0$.

For $1\le j\le m$, we have
\begin{align*}
|K_{j1}| & \le8j\|v^{(j)}\|_{2}\|u_{x}\|_{2}\|v^{(j)}\|_{2}\le8\|u\|_{3}\cdot j\|v^{(j)}\|_{2}^{2},\\
|L_{j1}| & \le8j\|v^{(j)}\|_{2}\|u^{(j)}\|_{2}\|v_{x}\|_{2}\le4\|v\|_{3}\left(j\|u^{(j)}\|_{2}^{2}+j\|v^{(j)}\|_{2}^{2}\right),
\end{align*}
and 
\begin{align}
\left|\sum_{j=1}^{m}\frac{e^{2\sigma j}}{j!^{2}}K_{j1}\right|\le & 8\|u\|_{3}\partial_{\sigma}\Phi_{\sigma,m}^{(2)}(v)\le8\|u\|_{3}\partial_{\sigma}\Phi_{\sigma,m}(u,v),\allowdisplaybreaks\label{eq:12201}\\
\left|\sum_{j=1}^{m}\frac{e^{2\sigma j}}{j!^{2}}L_{j1}\right|\le & 8e^{2\sigma}\|v\|_{3}\Phi_{\sigma,m}^{(1)}(u)+4e^{2\sigma}\|v\|_{3}\partial_{\sigma}\Phi_{\sigma,m}^{(1)}(u)\label{eq:12202}\\
 & \quad+4\|v\|_{3}\partial_{\sigma}\Phi_{\sigma,m}^{(2)}(v)\nonumber \\
\le & 8e^{2\sigma}\|v\|_{3}\Phi_{\sigma,m}(u,v)+(4e^{2\sigma}+4)\|v\|_{3}\partial_{\sigma}\Phi_{\sigma,m}(u,v).\nonumber 
\end{align}
For $j\ge2$, we have 
\begin{align}
|K_{j2}|\le & 8\|v^{(j)}\|_{2}\sum_{\ell=2}^{j}\binom{j}{\ell}\left(\|u^{(\ell)}\|_{2}\|v^{(j-\ell+1)}\|_{1}+\|u^{(\ell)}\|_{1}\|v^{(j-\ell+1)}\|_{2}\right)\allowdisplaybreaks\label{eq:12241}\\
\le & 8\sum_{\ell=2}^{j}\binom{j}{\ell}\|u^{(\ell)}\|_{2}\|v^{(j)}\|_{2}\|v^{(j-\ell)}\|_{2}\allowdisplaybreaks\nonumber \\
 & +8\sum_{\ell=2}^{j}\binom{j}{\ell}\|u^{(\ell-1)}\|_{2}\|v^{(j)}\|_{2}\|v^{(j-\ell+1)}\|_{2},\allowdisplaybreaks\nonumber \\
|L_{j2}|\le & 8\|v^{(j)}\|_{2}\sum_{\ell=2}^{j}\binom{j}{\ell}\left(\|u^{(j-\ell+1)}\|_{1}\|v^{(\ell)}\|_{2}+\|u^{(j-\ell+1)}\|_{2}\|v^{(\ell)}\|_{1}\right)\allowdisplaybreaks\label{eq:12242}\\
\le & 8\|u\|_{2}\|v^{(j)}\|_{2}^{2}+8\|u^{(1)}\|_{2}\|v^{(j)}\|_{2}\|v^{(j-1)}\|_{2}\allowdisplaybreaks\nonumber \\
 & +8\sum_{\ell=2}^{j-1}\binom{j}{\ell}\|u^{(j-\ell)}\|_{2}\|v^{(j)}\|_{2}\|v^{(\ell)}\|_{2}\allowdisplaybreaks\nonumber \\
 & +8\sum_{\ell=2}^{j-1}\binom{j}{\ell}\|u^{(j-\ell+1)}\|_{2}\|v^{(j)}\|_{2}\|v^{(\ell-1)}\|_{2}.\nonumber 
\end{align}
Notice that $\sum_{\ell=2}^{j-1}=0$ if $j\le2$ in \eqref{eq:12242}.
As for the first term in the estimate \eqref{eq:12242} of $L_{j2}$,
we have 
\begin{equation}
\sum_{j=2}^{m}\frac{e^{2j\sigma}}{j!^{2}}\|u\|_{2}\|v^{(j)}\|_{2}^{2}\le2\|u\|_{2}\Phi_{\sigma,m}^{(2)}(v).\label{eq:201912191}
\end{equation}
As for the second term, since 
\[
\|u^{(1)}\|_{2}\|v^{(j)}\|_{2}\|v^{(j-1)}\|_{2}\le\frac{\|u^{(1)}\|_{2}}{2}\left(\|v^{(j)}\|_{2}^{2}+\|v^{(j-1)}\|_{2}^{2}\right),
\]
and 
\[
\sum_{j=2}^{m}\frac{e^{2j\sigma}}{j!^{2}}\|u^{(1)}\|_{2}\|v^{(j-1)}\|_{2}^{2}\le e^{2\sigma}\|u\|_{3}\sum_{k=1}^{m-1}\frac{e^{2k\sigma}}{k!^{2}}\|v^{(k)}\|_{2}^{2}\le2e^{2\sigma}\|u\|_{3}\Phi_{\sigma,m}^{(2)}(v),
\]
we have 
\begin{equation}
\sum_{j=2}^{m}\frac{e^{2j\sigma}}{j!^{2}}\|u^{(1)}\|_{2}\|v^{(j)}\|_{2}\|v^{(j-1)}\|_{2}\le(1+e^{2\sigma})\|u\|_{3}\Phi_{\sigma,m}^{(2)}(v).\label{eq:201912192}
\end{equation}
Now we estimate the four sums in \eqref{eq:12241} and \eqref{eq:12242}.
Set $a_{k}=e^{k\sigma}\|u^{(k)}\|_{2}/k!\,(k=1,2,\dots,m+1)$, $b_{k}=e^{k\sigma}\|v^{(k)}\|_{2}/k!\,(k=0,1,\dots,m)$.
Then by Proposition \ref{prop:cauchyschwarz}, we get 
\begin{align*}
 & \sum_{j=2}^{m}\frac{e^{2j\sigma}}{j!^{2}}\sum_{\ell=2}^{j}\binom{j}{\ell}\|u^{(\ell)}\|_{2}\|v^{(j)}\|_{2}\|v^{(j-\ell)}\|_{2}=\sum_{j=2}^{m}\sum_{\ell=2}^{j}a_{\ell}b_{j}b_{j-\ell}\le\frac{\pi}{\sqrt{6}}\tilde{A}B\tilde{B},\allowdisplaybreaks\\
 & \sum_{j=2}^{m}\frac{e^{2j\sigma}}{j!^{2}}\sum_{\ell=2}^{j}\binom{j}{\ell}\|u^{(\ell-1)}\|_{2}\|v^{(j)}\|_{2}\|v^{(j-\ell+1)}\|_{2}\allowdisplaybreaks\\
 & =\sum_{j=2}^{m}\sum_{\ell=2}^{j}\frac{j-\ell+1}{\ell}a_{\ell-1}b_{j}b_{j-\ell+1}\le\frac{\pi}{\sqrt{6}}A\tilde{B}^{2},\allowdisplaybreaks\\
 & \sum_{j=3}^{m}\frac{e^{2j\sigma}}{j!^{2}}\sum_{\ell=2}^{j-1}\binom{j}{\ell}\|u^{(j-\ell)}\|_{2}\|v^{(j)}\|_{2}\|v^{(\ell)}\|_{2}=\sum_{j=0}^{m}\sum_{\ell=2}^{j-1}a_{j-\ell}b_{j}b_{\ell}\le\frac{\pi}{\sqrt{6}}A\tilde{B}^{2},\allowdisplaybreaks\\
 & \sum_{j=3}^{m}\frac{e^{2j\sigma}}{j!^{2}}\sum_{\ell=2}^{j-1}\binom{j}{\ell}\|u^{(j-\ell+1)}\|_{2}\|v^{(j)}\|_{2}\|v^{(\ell-1)}\|_{2}\allowdisplaybreaks\\
 & =\sum_{j=3}^{m}\sum_{\ell=2}^{j-1}\frac{j-\ell+1}{\ell}a_{j-\ell+1}b_{j}b_{\ell-1}\le\frac{\pi}{\sqrt{6}}\tilde{A}B\tilde{B}.
\end{align*}
Therefore by \eqref{eq:12241} , \eqref{eq:12242}, \eqref{eq:201912191},
\eqref{eq:201912192} and Proposition \ref{prop:A and tilde A}, 
\begin{align}
\left|\sum_{j=1}^{m}\frac{e^{2j\sigma}}{j!^{2}}K_{j2}\right|\le & \frac{8\pi}{\sqrt{6}}\left(\tilde{A}B\tilde{B}+A\tilde{B}^{2}\right)\allowdisplaybreaks\label{eq:12203}\\
\le & \frac{8\pi}{\sqrt{6}}(2+\sqrt{2})e^{\sigma}\sqrt{\Phi_{\sigma,m}(u,v)}\partial_{\sigma}\Phi_{\sigma,m}(u,v),\allowdisplaybreaks\nonumber \\
\left|\sum_{j=1}^{m}\frac{e^{2j\sigma}}{j!^{2}}L_{j2}\right|\le & (24+8e^{2\sigma})\|u\|_{3}\Phi_{\sigma,m}(u,v)\allowdisplaybreaks\label{eq:12204}\\
 & +\frac{8\pi}{\sqrt{6}}(2+\sqrt{2})e^{\sigma}\sqrt{\Phi_{\sigma,m}(u,v)}\partial_{\sigma}\Phi_{\sigma,m}(u,v).\nonumber 
\end{align}
Combining \eqref{eq:12201}, \eqref{eq:12202}, \eqref{eq:12203}
and \eqref{eq:12204}, we obtain
\begin{align}
\left|\sum_{j=1}^{m}\frac{e^{2j\sigma}}{j!^{2}}J_{3}\right|\le & \left[(24+8e^{2\sigma})\|u\|_{3}+8e^{2\sigma}\|v\|_{3}\right]\Phi_{\sigma,m}(u,v)\allowdisplaybreaks\label{eq:3rdtermconclusion}\\
 & +\frac{16\pi}{\sqrt{3}}(\sqrt{2}+1)e^{\sigma}\sqrt{\Phi_{\sigma,m}(u,v)}\partial_{\sigma}\Phi_{\sigma,m}(u,v)\allowdisplaybreaks\nonumber \\
 & +\left[8\|u\|_{3}+(4e^{2\sigma}+4)\|v\|_{3}\right]\partial_{\sigma}\Phi_{\sigma,m}(u,v).\nonumber 
\end{align}

\section*{Appendix}
\begin{prop}
\label{prop:cauchyschwarz}For non-negative numbers $a_{j}\,(j=1,\dots,m+1)$,
and $b_{j}\,(j=0,\dots,m)$, set $A=\left(\sum_{j=1}^{m+1}a_{j}^{2}\right)^{1/2}$,
$\tilde{A}=\left(\sum_{j=2}^{m+1}ja_{j}^{2}\right)^{1/2}$, $B=\left(\sum_{j=0}^{m}b_{j}^{2}\right)^{1/2}$,
$\tilde{B}=\left(\sum_{j=1}^{m}jb_{j}^{2}\right)^{1/2}$. Then we
have%
\begin{align}
 & \sum_{j=3}^{m+1}\sum_{\ell=2}^{j-1}a_{j}a_{\ell}a_{j-\ell}\le\frac{\pi}{\sqrt{6}}A\tilde{A}^{2},\allowdisplaybreaks\label{eq:AB1}\\
 & \sum_{j=3}^{m+1}\sum_{\ell=2}^{j-1}\frac{j-\ell+1}{\ell}a_{j}a_{\ell-1}a_{j-\ell+1}\le\frac{\pi}{\sqrt{6}}A\tilde{A}^{2},\allowdisplaybreaks\label{eq:AB2}\\
 & \sum_{j=2}^{m+1}\sum_{\ell=1}^{j-1}\frac{1}{j}a_{j}b_{\ell}b_{j-\ell-1}\le\frac{\pi}{\sqrt{6}}\tilde{A}B\tilde{B,}\allowdisplaybreaks\label{eq:AB3}\\
 & \sum_{j=2}^{m}\sum_{\ell=2}^{j}a_{\ell}b_{j}b_{j-\ell}\le\frac{\pi}{\sqrt{6}}\tilde{A}B\tilde{B,}\allowdisplaybreaks\label{eq:AB4}\\
 & \sum_{j=2}^{m}\sum_{\ell=2}^{j}\frac{j-\ell+1}{\ell}a_{\ell-1}b_{j}b_{j-\ell+1}\le\frac{\pi}{\sqrt{6}}A\tilde{B}^{2},\allowdisplaybreaks\label{eq:AB5}\\
 & \sum_{j=3}^{m}\sum_{\ell=2}^{j-1}a_{j-\ell}b_{j}b_{\ell}\le\frac{\pi}{\sqrt{6}}A\tilde{B}^{2},\allowdisplaybreaks\label{eq:AB6}\\
 & \sum_{j=3}^{m}\sum_{\ell=2}^{j-1}\frac{j-\ell+1}{\ell}a_{j-\ell+1}b_{j}b_{\ell-1}\le\frac{\pi}{\sqrt{6}}\tilde{A}B\tilde{B.}\label{eq:AB7}
\end{align}
\begin{proof}
We have 
\begin{align*}
 & \sum_{j=3}^{m+1}\sum_{\ell=2}^{j-1}a_{j}a_{\ell}a_{j-\ell}=\sum_{\ell=2}^{m}\sum_{j=\ell+1}^{m+1}a_{j}a_{\ell}a_{j-\ell}\le\sum_{\ell=2}^{m}\frac{\sqrt{\ell}}{\ell}a_{\ell}\sum_{j=\ell+1}^{m+1}\sqrt{j}a_{j}\cdot a_{j-\ell}\\
 & \le A\tilde{A}\sum_{\ell=2}^{m}\frac{1}{\ell}\cdot\sqrt{\ell}a_{\ell}\le\frac{\pi}{\sqrt{6}}A\tilde{A}^{2},\allowdisplaybreaks\\
 & \sum_{j=3}^{m+1}\sum_{\ell=2}^{j-1}\frac{j-\ell+1}{\ell}a_{j}a_{\ell-1}a_{j-\ell+1}=\sum_{\ell=2}^{m}\sum_{j=\ell+1}^{m+1}\frac{j-\ell+1}{\ell}a_{j}a_{\ell-1}a_{j-\ell+1}\\
 & \le\sum_{\ell=2}^{m}\frac{a_{\ell-1}}{\ell}\sum_{j=\ell+1}^{m}\sqrt{j}a_{j}\cdot\sqrt{j-\ell+1}a_{j-\ell+1}\le\tilde{A}^{2}\sum_{\ell=2}^{m}\frac{a_{\ell-1}}{\ell}\le\frac{\pi}{\sqrt{6}}A\tilde{A}^{2},\allowdisplaybreaks\\
 & \sum_{j=2}^{m+1}\sum_{\ell=1}^{j-1}\frac{1}{j}a_{j}b_{\ell}b_{j-\ell-1}\le\sum_{j=2}^{m+1}\sum_{\ell=1}^{j-1}\frac{\sqrt{\ell}\sqrt{j}}{\ell}a_{j}b_{\ell}b_{j-\ell-1}\\
 & =\sum_{\ell=1}^{m}\sum_{j=\ell+1}^{m+1}\frac{\sqrt{\ell}\sqrt{j}}{\ell}a_{j}b_{\ell}b_{j-\ell-1}\le\sum_{\ell=1}^{m}\frac{\sqrt{\ell}}{\ell}b_{\ell}\sum_{j=\ell+1}^{m+1}\sqrt{j}a_{j}\cdot b_{j-\ell-1}\le\frac{\pi}{\sqrt{6}}\tilde{A}B\tilde{B,}\allowdisplaybreaks\\
 & \sum_{j=2}^{m}\sum_{\ell=2}^{j}a_{\ell}b_{j}b_{j-\ell}\le\sum_{\ell=2}^{m}\frac{\sqrt{\ell}a_{\ell}}{\ell}\sum_{j=\ell}^{m}\sqrt{j}b_{j}b_{j-\ell}\le\frac{\pi}{\sqrt{6}}\tilde{A}B\tilde{B},\allowdisplaybreaks\\
 & \sum_{j=2}^{m}\sum_{\ell=2}^{j}\frac{j-\ell+1}{\ell}a_{\ell-1}b_{j}b_{j-\ell+1}\le\sum_{\ell=2}^{m}\frac{a_{\ell-1}}{\ell}\sum_{j=\ell}^{m}\sqrt{j}b_{j}\cdot\sqrt{j-\ell+1}b_{j-\ell+1}\le\frac{\pi}{\sqrt{6}}A\tilde{B}^{2},\allowdisplaybreaks\\
 & \sum_{j=3}^{m}\sum_{\ell=2}^{j-1}a_{j-\ell}b_{j}b_{\ell}\le\sum_{\ell=2}^{m-1}\frac{\sqrt{\ell}b_{\ell}}{\ell}\sum_{j=\ell+1}^{m}\sqrt{j}b_{j}\cdot a_{j-\ell}\le\frac{\pi}{\sqrt{6}}A\tilde{B}^{2},\allowdisplaybreaks\\
 & \sum_{j=3}^{m}\sum_{\ell=2}^{j-1}\frac{j-\ell+1}{\ell}a_{j-\ell+1}b_{j}b_{\ell-1}\le\sum_{\ell=2}^{m-1}\frac{b_{\ell-1}}{\ell}\sum_{j=\ell+1}^{m}\sqrt{j}b_{j}\cdot\sqrt{j-\ell+1}a_{j-\ell+1}\allowdisplaybreaks\\
 & \le\frac{\pi}{\sqrt{6}}\tilde{A}B\tilde{B}.\qedhere
\end{align*}
\end{proof}
\end{prop}

\,
\begin{prop}
\label{prop:A and tilde A}If $a_{k}=e^{k\sigma}\|u^{(k)}\|_{2}/k!$
$(k=1,2,\dots,m+1)$, $b_{k}=e^{k\sigma}\|v^{(k)}\|_{2}/k!$ $(k=0,1,\dots,m)$,
and $A,\tilde{A,}B,\tilde{B}$ are defined as in Proposition \ref{prop:cauchyschwarz},
then 
\begin{align*}
 & e^{-2\sigma}A\tilde{A}^{2}\le2\sqrt{2}e^{\sigma}\sqrt{\Phi_{\sigma,m}^{(1)}(u)}\,\partial_{\sigma}\Phi_{\sigma,m}^{(1)}(u)\\
 & \tilde{A}B\tilde{B}\le2e^{\sigma}\sqrt{\Phi_{\sigma,m}^{(2)}(v)}\,\sqrt{\partial_{\sigma}\Phi_{\sigma,m}^{(1)}(u)\cdot\partial_{\sigma}\Phi_{\sigma,m}^{(2)}(v)},\\
 & A\tilde{B}^{2}=\sqrt{2}e^{\sigma}\sqrt{\Phi_{\sigma,m}^{(1)}(u)}\,\partial_{\sigma}\Phi_{\sigma,m}^{(2)}(v).
\end{align*}
\end{prop}

\begin{proof}
We have 
\begin{align*}
 & e^{-2\sigma}A^{2}=2\Phi_{\sigma,m}^{(1)}(u),\\
 & e^{-2\sigma}\tilde{A}^{2}=\sum_{k=2}^{m+1}\frac{ke^{2(k-1)\sigma}}{k!^{2}}\|u^{(k)}\|_{2}^{2}\le\sum_{k=2}^{m+1}\frac{2(k-1)e^{2(k-1)\sigma}}{k!^{2}}\|u^{(k)}\|_{2}^{2}=2\partial_{\sigma}\Phi_{\sigma,m}^{(1)}(u).
\end{align*}
 Moreover, we have
\[
B^{2}=2\Phi_{\sigma,m}^{(2)}(v),\;\tilde{B}^{2}=\partial_{\sigma}\Phi_{\sigma,m}^{(2)}(v).\qedhere
\]
\end{proof}


\begin{thebibliography}{1}
\bibitem{BHP}Barostichi, R. F., A. A. Himonas and G. Petronilho.
``Autonomous Ovsyannikov theorem and applications to nonlocal evolution
equations and systems.'' \textsl{Journal of Functional Analysis }270
(2016): 330-358.

\bibitem{BHPpower}Barostichi, R. F., A. A. Himonas and G. Petronilho.
``The power series method for nonlocal and nonlinear evolution equations.''
\textsl{Journal of Mathematical Analysis and Applications} 443 (2016):
834-847.

\bibitem{BHPglobal}Barostichi, R. F., A. A. Himonas and G. Petronilho.
``Global analyticity for a generalized Camassa-Holm equation and
decay of the radius of spacial analyticity.'' \textsl{Journal of
Differential Equations }263 (2017): 732-764.

\bibitem{Boutet}Boutet de Monvel, A., A. Kostenko, D. Shepelsky and
G. Teschl. ``Long-time Asymptotics for the Camassa--Holm Equation.''
\textsl{SIAM Journal on Mathematical Analysis }41 (4) (2009): 1559--1588.

\bibitem{CamassaHolm}Camassa, R. and D. Holm. ``An integrable shallow
water equation with peaked solitons.'' \textsl{Physical Review Letters}\textit{
}71 (11) (1993): 1661-1664.

\bibitem{ChenLiu}Chen, R. M. and Y. Liu. ``Wave breaking and global
existence for a generalized two-component Camassa-Holm system.''
\textsl{International Mathematics Research Notices} 6 (2011): 1381-1416.

\bibitem{Constantin}Constantin, A. ``On the scattering problem for
the Camassa-Holm equation.'' \textsl{Proceeding of the Royal Society
A: Mathematical, Physical and Engineering Sciences} 457 (2008) (2001):
953-{}-970.

\bibitem{Constantin-Escher}Constantin, A. and J. Escher. ``Analyticity
of periodic traveling free surface water waves with vorticity.''
\textit{Annals of Mathematics }173 (2011): 559-568.

\bibitem{FokasFuchssteiner}Fokas, A. and B. Fuchssteiner. ``Symplectic
structures, their B\"acklund transformations and hereditary symmetries.''
\textsl{\emph{Physica D: Nonlinear Phenomena}}\emph{ }4 (1) (1981):
47-66.

\bibitem{He-Yin}He, H. and Z. Yin. ``The global Gevrey regularity
and analyticity of a two-component shallow water system with higher-order
intertia operators, \textsl{Journal of Differential Equations} 267
(2019): 2531-2559. 

\bibitem{HM}Himonas, A. A. and G. Misio\l{}ek. ``Analyticity of
the Cauchy problem for an integrable evolution equation.'' \textsl{\emph{Mathematische
Annalen}} 327 (2003): 575-584.

\bibitem{KatoMasuda}Kato, T. and K. Masuda. ``Nonlinear evolution
equations and analyticity. I.'' \textsl{Annales Henri Poincar\'e C, Analyses Non Lin\'eaire}
3 (6) (1986): 455-467.

\bibitem{Kato-Ponce}Kato, T. and G. Ponce. ``Commutator estimates
and the Euler and Navier-Stokes equations.'' \textsl{\emph{Communications
on Pure and Applied Mathematics}}\emph{ } 41 (1988): 891-907.

\bibitem{Komatsu}Komatsu, H. ``A characterization of real analytic
functions.'' \textsl{Proceedings of the Japan Academy, Ser. A Mathematical
Sciences }36 (1960): 90-93.

\bibitem{Kotake}Kotake, T. and M. S. Narasimhan. ``Regularity theorems
for fractional powers of a linear elliptic operator.'' \textsl{Bulletin de la Soci\'et\'e Math\'ematique de France}
90 (1962): 449--471.

\bibitem{Parker}Parker, A. ``On the Camassa-Holm equation and a
direct method of solution I. Bilinear form and solitary waves.''
\textsl{Proceeding of the Royal Society A: Mathematical, Physical
and Engineering Sciences }4 60 (2004): 2929-2957. 

\bibitem{Shabat-Alonso}Shabat, A. and L. Mart{\'i}nez Alonso. ``On
the prolongation of a hierarchy of hydrodynamic chains.'' In \textsl{Proceedings
of the NATO Advanced Research Workshop, Cadiz, Spain 2002}, 263-280.
NATO Science Series. Dordrecht: Kluwer, 2004.

\bibitem{yamane muCH}Yamane, H. ``Local and global analyticity for
$\mu$-Camassa-Holm equations.`` arXiv 1906.11411 {[}math AP{]}.

\bibitem{ZY}Zhang, Z. and Z. Yin. ``Global existence for a two-component
Camassa-Holm system with an arbitrary smooth function.'' \textsl{Discrete
and Continuous Dynamical Systems }38 (11) (2018): 5523-5536.
\end{thebibliography}
\end{document}